\documentclass[a4paper,10pt]{article}
\usepackage[margin=1in]{geometry}
\usepackage{amsmath, amsthm, amssymb}
\usepackage{comment}
\usepackage{hyperref}
\usepackage[T1]{fontenc}
\usepackage[utf8]{inputenc}
\usepackage{graphicx}
\usepackage[active]{srcltx}
\usepackage{enumerate}
\hypersetup{colorlinks=true, pdfstartview=FitV, linkcolor=blue, citecolor=blue, urlcolor=blue}

\usepackage{enumerate}
\usepackage{graphicx}
\usepackage[hang,small]{caption}

\newtheorem{thm}{Theorem}[section]

\newtheorem{lem}[thm]{Lemma}

 \newtheorem{Rmk}[thm]{Remark}

  \newtheorem{Def}[thm]{Definition}

 \def\N {\mathbb{N}}
\def\R {\mathbb{R}}

\def\cR {\mathcal{R}}

\def\de {{\partial}}
\def\eps {{\epsilon}}

\newcommand{\supp}{\operatorname{supp}}
\newcommand{\be}{\begin{equation}}
\newcommand{\ee}{\end{equation}}

\newcommand{\by}{\bold{y}}

\numberwithin{equation}{section}

\begin{document}


\title{Non Existence and Strong Ill-Posedness in $H^2$\\ for the Stable IPM Equation}
\date{}
\author{Roberta Bianchini\footnote{Consiglio Nazionale Delle Ricerche, 00185, Rome, Italy. e-mail: roberta.bianchini@cnr.it}, \; Diego C\'{o}rdoba\footnote{Instituto de Ciencias Matem\'aticas CSIC-UAM-UCM-UC3M, Spain. e-mail: dcg@icmat.es}, \; Luis Mart\'{\i}nez-Zoroa \footnote{University of Basel, Switzerland. e-mail: †
luis.martinezzoroa@unibas.ch}}
\maketitle 
\begin{abstract}
    We prove the non-existence and strong ill-posedness of the Incompressible Porous Media (IPM) equation for initial data that are small $H^2(\mathbb{R}^2)$ perturbations of the linearly stable profile $-x_2$. A remarkable novelty of the proof is the construction of an $H^2$ perturbation, which solves the IPM equation and neutralizes the stabilizing effect of the background profile near the origin, where a strong deformation leading to non-existence in $H^2$ is created. This strong deformation is achieved through an iterative procedure inspired by the work of C\'{o}rdoba and Mart\'{\i}nez-Zoroa (Adv. Math. 2022). However, several differences - beyond purely technical aspects - arise due to the anisotropic and, more importantly, to the partially dissipative nature of the equation, adding further challenges to the analysis.
\end{abstract}


\section{Introduction}
The Incompressible Porous Media (IPM) system in two space dimensions is an active scalar equation given by:
\begin{equation}\label{eq:IPMsystem}
\begin{cases}
    \partial_t \rho + \textbf{u} \cdot \nabla \rho = 0, \\
    \textbf{u} = - \kappa \nabla P + \textbf{g}\rho, \quad \textbf{g} = (0, -g)^\top, \quad \text{(Darcy's law)} \\
    \nabla \cdot \textbf{u} = 0,
\end{cases}
\end{equation}
modeling the dynamics of a fluid of density $\rho = \rho(x_1, x_2, t): \mathbb{R}^2\times \R_+ \rightarrow \mathbb{R}$ through a porous medium according to Darcy's law. Here, $\kappa > 0$ and $g > 0$ denote the permeability coefficient and the gravitational acceleration, respectively. For simplicity, we assume $\kappa = g = 1$ in the subsequent analysis.
We refer to \cite{castro2} and references therein for further explanations on the physical background and applications of this model.
The active scalar velocity $\textbf{u}=(u_1, u_2)$ of system \eqref{eq:IPMsystem} can be reformulated in terms of 
a singular integral operator of degree 0 as follows
\begin{align}\label{eq:ipm}\tag{IPM}
    \de_t \rho + \mathbf{u} \cdot \nabla \rho&=0, \\
    \mathbf{u} &=(u_1,u_2)=-(\cR_2, -\cR_1) \cR_1 \rho,\notag
\end{align}
where $\cR_1, \cR_2$ denote the first and the second component of the Riesz transform in two space dimensions 
\begin{align}\label{eq:riesz}
    \cR_1=(-\Delta)^{-1/2} \partial_{x_1}, \qquad  \cR_2=(-\Delta)^{-1/2} \partial_{x_2}.
\end{align}
and can be written in terms of convolutions kernels in the two following equivalent ways - up to an integration by parts, see for instance \cite{cordoba2006} - 
\begin{align}\label{eq:vel-kernel}
    \mathbf{u}(x, \cdot)=&  \text{PV} (H \star \rho)(x, \cdot) = \text{PV} ((H_1 \star \rho)(x, \cdot), (H_2 \star \rho)(x, \cdot))- \frac 12 (0,\rho(x, \cdot)),
\end{align}
where
\begin{align}\label{eq:kernels}
    (H_1(x), H_2(x))&= \frac{1}{2\pi} \left(- 2 \frac{x_1 x_2}{|x|^4}, \frac{x_1^2-x_2^2}{|x|^4}\right),
\end{align}
and
\begin{align}\label{eq:vel-kernel2}
     \mathbf{u}(x, \cdot)&=(K \star \nabla^\perp \rho)(x, \cdot)=(-(K \star \de_{x_2}\rho)(x, \cdot), (K \star \de_{x_1}\rho)(x, \cdot)),
\end{align}
with
\begin{align}
    K(x)={-\frac{1}{2\pi}} \frac{x_1}{|x|^2}.
\end{align}
In this work, we study the IPM equation in the stable regime, specifically the system \eqref{eq:IPMsystem} near the spectrally stable steady state profile
\begin{align}\label{eq:strat}
    \bar\rho_{\text{stable}}(x) = -x_2.
\end{align}
We will say that $\rho_{\text{stable}}(x,t)$ is a solution to the \emph{stable} IPM equation if it fulfills
\begin{equation}\label{eq:ipm-stable}
    \de_t \rho_{\text{stable}} +\bold u [\rho_{\text{stable}}] \cdot \nabla \rho_{\text{stable}}-u_2[ \rho_{\text{stable}}]=0.
\end{equation}
Note that this is equivalent to study the solution to \eqref{eq:ipm} with (stable) initial conditions 
\begin{align}
    \label{eq:data-stable}
     \rho(x,0)=-x_{2}+\rho_{\text{stable}}(x,0).
\end{align}
Furthermore, we will only consider classical solutions to the stable IPM equation belonging to $C^1$.
%
%
\begin{Rmk}
    It can be observed from the proof that any sufficiently regular profile \( g(x_2) \) with \( g'(x_2) < 0 \) could also be considered. The non-trivial challenge in extending the results to this more general case lies in adapting the construction of the initial data, as discussed below in Lemma \ref{initialdata}. 
\end{Rmk}
\subsection{Motivations and existing literature}
A fundamental challenge in mathematical physics is to understand
the behavior of active scalar equations. This is particularly true when it comes to the study of equations where the operator relating the velocity and the active scalar
is a singular integral operator of zero order, as in the Constantin-Lax-Majda equation \cite{constantin}, the De Gregorio model \cite{degreg1}, the SQG equation \cite{cordobasqg2022} and, for systems, the Boussinesq equations \cite{bianchini24}.
\paragraph{Ill-posedness at critical regularity.}
While there have been several results of ill-posedness of fluid-dynamics equations at super-critical regu\-larity, here we are rather interested in \emph{ill-posedness at critical regularity}. In critical $L^\infty$-based spaces, equations whose transport velocity is represented by singular integral operators of zero order are very likely to be strongly ill-posed, due to the unboudedness of such singular integral operators in $L^\infty$. In the case of the critical space $W^{1, \infty}$ for equations posed in $\R^2$, this was showed in \cite{kim24} for the SQG equation, in \cite{khalil} for the Riesz Transform Problem and in \cite{bianchini24} for the Boussinesq equations. In critical $L^2$-based spaces like $H^2(\R^2)$, the mechanism leading to ill-posedness can be more subtle as singular integral operators of zero order behaves well in such ambient, see \cite{cordobasqg2022, kim24} for the ill-posedness of the SQG equation. Since there are great similarities between the inviscid versions of the SQG and IPM equations, it is very reasonable to expect that the same results regarding non-existence and strong ill-posedness in $H^2(\mathbb{R}^2)$ for SQG would also hold for the IPM equation.
This expectation holds and, in particular, solutions with a strong deformation at the origin leading to our result of ill-posedness are constructed via an iteration method that is inspired by the recent work \cite{cordobasqg2022} on the SQG equation, by the second and third author.
However, despite these similarities, important differences between the SQG and the IPM equations, such as symmetries and (an-)isotropy, exist. A key distinguishing feature of the \emph{stable} IPM equation treated in this work, which necessitates, among the other things, a fundamentally different initialization of the iteration procedure employed here, is highlighted below.

\paragraph{Ill-posedness in $L^2$-based critical space despite the presence of dissipation.}
An important point to stress is that the result of ill-posedness in $H^2(\mathbb{R}^2)$ presented in this paper holds for the \emph{stable} \eqref{eq:ipm} equation, specifically for perturbations of the spectrally stable steady state \eqref{eq:strat}. This is a crucial novelty of this work, where ill-posedness is demonstrated in an $L^2$-based critical space, namely $H^2(\mathbb{R}^2)$, despite the presence of a linear dissipative operator in the equations. 

In fact, in the vicinity of a spectrally stable steady state (under the assumption $\partial_{x_2}\bar{\rho}_{\text{stable}}(x)<0$), one might expect solutions to \eqref{eq:ipm} to exhibit better behavior than in the general case. This expectation is motivated by the \emph{partial dissipation} induced by the stable steady state.

Setting $\rho_{pert}(x, t)=\rho(x, t)+x_2$, where $\rho(x,t)$ solves the IPM equation with initial data \eqref{eq:data-stable}, $\rho_{pert}(x, t)$ satisfies
\begin{align}
    \de_t \rho_{pert} + \mathbf{u}[\rho_{pert}] \cdot \nabla \rho_{pert} &= u_2[\rho_{pert}],\notag
\end{align}
where
$
    u_2[\rho_{pert}]=\cR_1^2 \rho_{pert}=(-\Delta)^{-1} \de_{x_1}^2\rho_{pert}
$ encodes (partial, anisotropic) dissipation (in the horizontal direction of the Fourier space) as
\begin{align*}
    (u_2[\rho_{pert}], \rho_{pert})_{L^2}=((-\Delta)^{-1}\de_{x_1}^2 \rho_{pert}, \rho_{pert})_{L^2}=-\|(-\Delta)^{-\frac{1}{2}}\de_{x_1}\rho_{pert}\|_{L^2}^2.
\end{align*}
In the critical space $W^{1, \infty}$, the effect of this partial dissipation is not able to prevent strong ill-posedness, as proved in \cite{bianchini24} for the Boussinesq equations (and the very same argument can be applied to the stable IPM equation).
However, in energy-based space as $H^2$, the stabilizing effect of the stable profile could potentially have an impact on solutions' behavior. In this work, we disprove this intuition for the IPM equation, showing that it is strongly ill-posed even for small $H^2$ perturbation of a spectrally stable profile.

\paragraph{Small-scale formation.}
Recently, solutions to the IPM equation exhibiting infinite-in-time growth of derivatives were constructed in \cite{kiselev2021}. In light of the results of the present work, we can complement some of the findings reported in \cite{kiselev2021}. Specifically, in Theorem 1.1 in \cite{kiselev2021} we can deduce that the regularity parameter $s$ must be strictly greater than 2; otherwise, a solution may not exist. 

\paragraph{Well-posedness and stability results.}
We conclude this brief outline by quoting, on the positive side, well-posedness results for the IPM equation in sub-critical regularity. The local existence and uniqueness for initial data
$\rho(x,0)=\rho_{\text{in}}(x) \in C^{1,\alpha}(\R^2)$ with $0<\alpha<1$ was shown in \cite{cordoba2006} and the analogous result can be shown to be hold in sub-critical Sobolev spaces $H^{2+\epsilon}(\R^2)$, for any $\epsilon>0$. Due to aforementioned partially dissipative nature of stable IPM, the long-time behavior of such equation is particularly interesting. Until now, the best available result proves the global-in-time existence of small perturbations of the stable profile for the IPM equation in $H^{3+\epsilon}(\R^2)$ ($\epsilon>0$), as shown in the work of the first author with Crin-Barat and Paicu \cite{bianchini23}; see also \cite{kim1} and the first result in this direction by Elgindi in \cite{elgindi-ipm} and \cite{castro2} for the strip (see also \cite{Park2024}). It is very likely that the asymptotic stability holds in $H^{2+\epsilon}(\R^2)$, for any $\epsilon>0$, \cite{BJPW24}. In particular, this last observation motivates the relevance of the work of the present paper, which proves that lowering the regularity assumption to $H^2(\R^2),$ the stability of the equations near the stable profile fails in the worst possible way, as solutions blow up at any instant of time.

\paragraph{Future directions.}
It was recently discovered in \cite{bianchini23} that the stable IPM equation can be recovered as a relaxation limit of the two-dimensional Boussinesq equations with damped velocity, near the stable profile \eqref{eq:strat}. For the damped Boussinesq equations, the asymptotic stability of the steady state \eqref{eq:strat} is known to hold in $H^{3+}$, \cite{bianchini23}, and, as with the IPM equation, at least some quantitative long-time stability is expected to hold in $H^{2+}$. 

We believe that an extension of our proof from the scalar equation to the case of systems can be applied to demonstrate non-existence and strong ill-posedness in $H^2(\mathbb{R}^2)$ for the damped two-dimensional Boussinesq equations near a stable profile in \cite{bianchini23}. As a byproduct, we expect to recover the same result for the two-dimensional Boussinesq equations near a stable profile (without damping in the velocity equation). This result would be very interesting \emph{per se}, as the understanding of stable dynamics for the two-dimensional Boussinesq equations remains a widely open problem, with recent advancements in \cite{wid2024}.

\subsection{Main result}
Our main result reads as follows.
\begin{thm}[Non existence and strong ill-posedness in $H^2$ for \eqref{eq:ipm} near stable profile]\label{thm:main}
Given $\epsilon>0$, 
we can find $T_{\epsilon}>0$ and a function $\rho_{\text{in}}(x)$ with $\|\rho_{\text{in}}\|_{H^2(\R^2)}\leq \epsilon$ such that, for $t\in[0,T_{\epsilon}]$, there exists a classical solution $\rho(x,t)$ to the  \eqref{eq:ipm} equation with ``stable'' initial data 
\begin{align*}
    \rho(x, 0)=-x_2 + \rho_{\text{in}}(x)
\end{align*}
satisfying 
\begin{align*}
\|\rho(t)+x_2\|_{H^2(\R^2)}(t)=\infty, \quad \text{for any} \quad t\in(0,T_{\epsilon}],
\end{align*}
while 
\begin{align*}
\|\rho(t)+x_2\|_{H^{2-\eps}(\R^2)}(t)<C, \quad \text{for some} \quad C>0, \quad \text{for any} \quad \eps>0, \quad \text{for any} \quad t\in(0,T_{\epsilon}].
\end{align*}

Furthermore, $\rho(x,t)$ is the only solution with initial conditions $\rho(x,0)$ such that $\rho(x,t)+x_2\in H^1(\R^2)$ for $t\in[0,a]$ and for any $0 < a \le T_{\epsilon}$. 
\end{thm}

\begin{Rmk}
    Notice that the initial data $\rho(x,0)$ in Theorem \ref{thm:main} has \emph{infinite energy}. This is the reason why the blow-up at any instant is proved for the perturbation $\|\rho (t) + x_2\|_{H^2}=\infty$.
\end{Rmk}
\begin{Rmk}
    Even though the solutions constructed $\rho(x,t)+x_{2}$ do not stay bounded in $H^2$, they do stay in $C^1$ for the whole time interval considered.
\end{Rmk}
\begin{Rmk}[Non existence and strong ill-posedness in $H^2$ for IPM near equilibrium]
    Our result also applies to the \eqref{eq:ipm} equation near the equilibrium $\bar{\rho}_{\text{eq}} = 0$. This can be easily verified by revisiting the proof and simplifying or omitting certain steps. However, we choose to state the main result for the \emph{stable} IPM equation \eqref{eq:ipm-stable}, as this setting presents the most significant challenges.
\end{Rmk}

\subsection{Ideas of the proof}
In order to construct our rapidly growing solutions, we would like to first find solutions $\rho(x,t)$ to the stable IPM equation \eqref{eq:ipm-stable} that generate a strong hyperbolic flow around the origin, say
$$\left|\int_{0}^{1}\partial_{x_{1}}u_{1}[\rho](x=0,s)\, ds\right|\geq M$$
for arbitrarily big $M$.

If we then add a perturbation  $\rho_{pert}(x,0)$ around the origin, we could expect the perturbation to fulfil approximately the evolution equation
$$\partial_{t}\rho_{pert}+(\partial_{x_{1}}u_{1}[\rho])(x=0,t)(x_{1},-x_{2})\cdot\nabla\rho_{pert}=0$$
which can be solved explicitly by

$$\rho_{pert}(x,t)=\rho_{pert}(\text{e}^{\int_{0}^{t}-\partial_{x_{1}}u_{1}[\rho](x=0,s)ds}x_{1},\text{e}^{\int_{0}^{t}\partial_{x_{1}}u_{1}[\rho](x=0,s)ds}x_{2}).$$

We could then choose $\rho_{pert}(x,0)$ appropriately to obtain the fast growth in $H^2$. Furthermore, if we can construct
$$\left|\int_{0}^{1}\partial_{x_{1}}u_{1}[\rho](x=0,s)\, ds\right|= \infty,$$
we could hope to use this infinitely strong deformation to obtain loss of regularity.

If we plan to use this strategy, there are two main difficulties that we need to overcome. First, we would need to be able to construct an infinitely strong hyperbolic velocity. Note that it is not enough to fix initial conditions such that
$$\partial_{x_{1}}u_{1}[\rho_{\text{in}}](x=0)=\infty,$$
as these initial conditions are too irregular to ensure that a solution exists, and even if it does, the behaviour of the solution could be wild enough that
$$\int_{0}^{1}\partial_{x_{1}}u_{1}[\rho_{\text{in}}](x=0,s)ds=M<\infty.\footnote{This is not only a theoretical possibility, one can, in fact, construct solutions where this kind of behaviour occurs.}$$

We need to consider a scenario where we can effectively study the behavior and gain meaningful control over the solution. If we were not focusing on the stable case, a reasonable starting point would be to consider

$$\rho_{\text{in}}(x)=\sum_{i=0}^{\infty}\frac{f(\lambda_{i}x)}{i}$$
where $\lambda_{i+1}\gg\lambda_{i}$, with $f(x)$ supported in the ring $2\geq|x|\geq 1$ and
$$\partial_{x_{1}}u_{1}[f(\cdot)](x=0)\neq 0.$$
These initial conditions have finite $H^2$ norm and generate an infinite hyperbolic deformation in the origin at $t=0$. Furthermore, formally taking the limit $\lambda_{i+1}\gg\lambda_{i}$ gives us a limit system that one can hope to study.

Moreover, by choosing $f(x)$ with the right properties, one can show that, if we study this formal limit system when $\lambda_{i+1}\gg\lambda_{i}$ and assume that it actually approximates the real solution, we have

$$\int_{0}^{t}\partial_{x_{1}}u_{1}[\rho](x=0,s)ds=\infty$$
for any $t>0$. This still leaves several issues:
\begin{itemize}
    \item We need to show that this formal approximation actually approximates the equation in a meaningful way. A big part of the technical difficulties comes from this point.
    \item We are interested in studying the stable IPM equation \eqref{eq:ipm-stable}, so we would need to consider the initial condition
    $$\rho_{\text{in}}(x)=\sum_{i=0}^{\infty}\frac{f(\lambda_{i}x)}{i}-x_{2},$$
    which is problematic for two reasons. First, this background $-x_{2}$ has a regularizing effect in the solutions, which is why small perturbations of $-x_{2}$ in $H^{3+\eps}$ are actually global in time, so this effect could oppose our infinite norm growth. Furthermore, for the approximation we intend to apply when $\lambda_{i+1} > \lambda_{i}$ to actually approximate the true evolution of the equation, it is necessary to decompose the solution into an infinite number of components with supports that are sufficiently "far" apart. However, due to dissipative effects caused by the background, this separation of supports does not hold for $t > 0$.
    \item Finally, the hyperbolic deformation we construct this way is only infinite at $x=0$, so any perturbation would need to be infinitely concentrated around the origin to feel its full effects. Furthermore, this background is extremely unstable, so most perturbations of the solution would completely change the behaviour.
\end{itemize}
To solve all these issues, we start by constructing a "hole" in our initial conditions, i.e., we find a compactly supported function $g(x)$ such that
$$g(x)-x_{2}=0$$
for some small ball around the origin. Since we want our initial conditions to be small in $H^2$, we choose $\|g\|_{H^2}\leq \eps.$ This allows us to cancel the stabilizing effects of the background as much as needed: it does not neutralize its full impact, but it does get rid of the most inconvenient terms for our construction.

Next, to actually construct the solution with a strong hyperbolic velocity, we use an iterative argument. Given $\rho_{i}(x,t)$ a solution to the IPM equation supported away form the origin, we can then add a new "layer", by studying the solution with initial conditions

$$\rho_{i}(x,0)+f(\lambda_{i+1} x)$$
and for very big $\lambda_{i+1}$ (depending on $\rho_{i}(x,t)$), we can obtain explicit bounds for the behaviour of the new solution. Using these bounds and making the right choice of $f(x)$, we can produce a velocity with the desired properties around the origin.

Similarly, we can also obtain the rapid norm growth by adding layers that grow rapidly in $H^2$ when under the influence of the very strong hyperbolic deformation generated by the previous layer.

Taking the limit when there is an infinite number of layers and denoting it by $\rho_{\infty}$, we can then show that this limit is a solution to the stable IPM equation with the desired properties.
\subsection{Conventions and notation}
As usual, we will exploit the symmetries of the equation. 
\begin{Def}\label{def:symmetric}
    A function $f(x,t)$ is called symmetric for $t\in[0,T]$ if $f(x_{1},x_{2},t)=f(-x_{1},x_{2},t),$ $f(x_{1},x_{2},t)=-f(x_{1},-x_{2},t)$,  namely $f(x_1, x_2, t)$ is \emph{even} in $x_1$ and \emph{odd} in $x_2$.
\end{Def}
\begin{Rmk}
    If the initial data $\rho(x,0)$ is symmetric and $\rho(x,t)$ is the  regular solution to \eqref{eq:ipm} with such initial data, then $\rho(x,t)$ is symmetric. Furthermore,  $\mathbf{u}[\rho](x=0,t)=0.$
\end{Rmk}
\begin{itemize}
    \item We use the symbol $\lesssim$ (resp. $\gtrsim$) to denote $\le C$ (resp. $\ge C$), where the constant $C>0$ is independent of the relevant parameters.
    \item The constant $C\in \R_+$ is generic and may change from line to line.
    \item The symbol $B_r(x_0)$ denotes a disk of radius $r>0$, centered in $x_0 \in \R^2$. 
\end{itemize}
\section{Construction of initial data with deformation at the origin and persistence of the deformation}
In order to construct a strong deformation at the origin, which will lead to strong ill-posedness and non existence in $H^2(\R^2)$, we need to design an initial perturbation that neutralizes the stabilizing effect of the (stable) background stratification \eqref{eq:strat}, at least near the origin. 
This is the content of the next result.
\subsection{Initial data and approximate solution in the stable setting}
\begin{lem}\label{initialdata}
    For any $0 < \eps_0 < 1$, we can construct a symmetric function $\rho_{\text{in}}(x)$ satisfying the following:
    \begin{enumerate}
        \item $\rho_{\text{in}} (x) \in C^\infty (\R^2)$  and $\supp (\rho_{\text{in}}) \subset  B_1(0)$; 
        \item there exists $0<\delta_0 < 1$ such that $\rho_{\text{in}}(x)=x_2$ for $x \in B_{\delta_0}(0)$;
        \item the $H^2$ norm is arbitrarily small: $\|\rho_{\text{in}}\|_{H^2} \le \eps_0$.
    \end{enumerate}
\end{lem}
\begin{proof}
    For some big $K>0$ (which will depend on $\eps_0$) and some small $a>0$ (which will depend on $K$) , the ansatz for the function $\rho_{\text{in}}(x)$ reads 
\begin{align}\label{eq:rho0}
    \rho_{\text{in}} (x) &= a \sum_{i=1}^K \frac{f(\lambda_i x)}{i \lambda_i } =a \sum_{i=1}^K \frac{f(\mu^i)}{i \lambda_i}=a \sum_{i=1}^K \frac{\varphi_i(x)}{i \lambda_i}, \qquad \mu^i:= \lambda_i x, \quad \varphi_i(x):=f(\lambda_i x), 
\end{align}
where $f(x) \in C^\infty(\R^2)$ is a compactly supported function with $\supp f \in B_1(0).$ 
Clearly, in order to satisfy the above properties, we need to characterize the function $f$ and the (finite) sequence $\{\lambda_i\}$. 
In particular, we take 
\begin{align*}
    \de_{\mu_2^i} f(\mu^i)=1 \quad \text{for} \quad \mu^i \in B_{\frac 12}(0),
\end{align*}
namely
\begin{align*}
    f(\mu^i)=\mu_2^i=\lambda_i x_2, \quad \mu^i \in B_{\frac 12}(0) \quad \text{for} \quad  x \in B_{\frac{1}{2\lambda_i}}(0). 
\end{align*}
About the finite sequence $\{\lambda_i\}$, let $\lambda_1>1$ and $\lambda_{i+1}> 2 \lambda_i$. We have that, for $x\in B_{\frac{1}{2\lambda_K}}$, for $i\in \{1, \cdots, K\} $
\begin{align*}
    f(\mu^i)=\lambda_i x_2 . 
\end{align*}
Then we see that, for $x \in B_{\frac{1}{2\lambda_K}},$
\begin{align*}
    \rho_{\text{in}}(x)=a \sum_{i=1}^K \frac{\lambda_i x_2}{i \lambda_{i}}= a \sum_{i=1}^K \frac{x_2}{i}= a x_2 \mathfrak{H}_K,  
\end{align*}
where $\mathfrak{H}_K$ is the $K$-th harmonic number, behaving like $\log (K)$. 
Choosing 
$$a=\frac{1}{\mathfrak{H}_K}, \quad \delta_0=\frac{1}{2\lambda_K},$$ we have that properties 1. and 2. are satisfied.  Notice also that this choice implies 
\begin{align*}
    \|\rho_{\text{in}}\|_{C^1}& \le a \|f\|_{C^1}\sum_{i=1}^K \frac{1}{i} \le \|f\|_{C^1}.
\end{align*}
Let us prove property 3. We first compute the $L^2$ and $\dot H^1$ norm. Using the Minkowski inequality, 
\begin{align*}
    \|\rho_{\text{in}}\|_{L^2}&=a \left(\int \left(\sum_{i=1}^K \frac{f(\lambda_i x)}{i \lambda_i }\right)^2 \, dx_1 \, dx_2\right)^\frac 12 \le a \sum_{i=1}^K \frac{\|f(\lambda_i x)\|_{L^2}}{i \lambda_i} \le a \|f\|_{L^2} \sum_{i=1}^K \frac{1}{i \lambda_i^2}.
\end{align*}
Notice that $\lambda_1>1$ and $\lambda_{i+1}> 2 \lambda_i$ implies, in particular, that $\lambda_2>2$ and, more generally, $\lambda_{i}>i$. Therefore, the sum $\sum_{i=1}^\infty \frac{\|f\|_{L^2}}{i \lambda_i^2}=M_0$ for some $M_0 >0$ and $\sum_{i=1}^K \frac{\|f\|_{L^2}}{i \lambda_i^2} < M_0 $ independent of $K$. We have that
\begin{align*}
    \|\rho_{\text{in}}\|_{L^2} \le a M_0 =  \frac{M_0}{\mathfrak{H}_K}.
\end{align*}
Similarly, we have that
\begin{align*}
    \|\rho_{\text{in}}\|_{\dot H^1}& \le a \sum_{i=1}^K \frac{\|f(\lambda_i x)\|_{\dot H^1}}{i}
    \le a \|f\|_{\dot H^1} \sum_{i=1}^K \frac{1}{i \lambda_i} \le \frac{M_1}{\mathfrak{H}_K},
\end{align*}
where, as before, since $\lambda_i > i$, the sum $\sum_{i=1}^K \frac{\|f\|_{H^1}}{i \lambda_i} \le M_1$ for some $\infty>M_1>0$.
Let us consider the second derivatives. 
Denoting by $D^2=\de_{x_{jk}}^2$, $j, k \in \{1, 2\}$ any second derivative, noticing that 
\begin{align*}
   D^2 \varphi_i(x)&=
   0 \quad \text{for} \quad |x| \le \frac{1}{2 \lambda_i} \quad \text{and} \quad |x| > \frac{1}{\lambda_i},
\end{align*}
namely 
\begin{align*}
    \supp D^2 (\varphi_i) \subset \mathcal{C}_{\lambda_i}:=B_{\frac{1}{\lambda_i}}\setminus B_{\frac{1}{2\lambda_i}},
\end{align*} and that $\lambda_{i+1}>2 \lambda_i$ implies that all the supports of $D^2 \varphi_i(x)$ are disjoint as 
we have that
\begin{align*}
    \|D^2\rho_{\text{in}}\|_{L^2}&= a \left( \int \left(\sum_{i=1}^K \frac{D^2 \varphi_i(x)}{i \lambda_i}\right)^2 \, dx_1 \, dx_2\right)^\frac{1}{2}= a \left( \int \left(\sum_{i=1}^K \frac{\lambda_i D^2 f(\mu^i)}{i }\right)^2 \, dx_1 \, dx_2\right)^\frac{1}{2}\\
    & = a \left( \sum_{i=1}^K \int_{\mathcal C_{\lambda_i}} \left( \frac{\lambda_i D^2 f(\mu^i)}{i }\right)^2 \, dx_1 \, dx_2\right)^\frac{1}{2}
    = a \left( \sum_{i=1}^K \frac{1}{i^2} \int_{\mathcal C_{\lambda_i}} |D^2 f(\mu^i)|^2 \, d\mu^i_1 \, d\mu^i_2\right)^\frac{1}{2}\\
    & \le a \|f\|_{H^2}\left(\sum_{i=1}^K \frac{1}{i^2}\right)^\frac{1}{2} \le a\|f\|_{H^2} \left(\sum_{i=1}^\infty \frac{1}{i^2}\right)^{\frac12} = a \|f\|_{H^2}\frac{\pi}{\sqrt 6}.
\end{align*}
Altogether, since we want $\|\rho_{\text{in}}\|_{H^2} \le \eps_0$, we require
\begin{align*}
    \|\rho_{\text{in}}\|_{H^2} \le \left(M_0+M_1+\frac{\pi}{\sqrt 6}\right)\frac{\|f\|_{H^2}}{\mathfrak{H}_K} \le \eps_0,
\end{align*}
which is satisfied by choosing $K$ big enough. This proves property 3 and ends the proof.
\end{proof}
Building on the above result, we can prove the next lemma, which specifically demonstrates the persistence of the sign of $\de_{x_1}u_1 < 0$ over time.
\begin{lem}\label{lem:stable-exact-sol}
    For any $0 < \eps_{0} < 1$, there exists a solution $\rho(x,t)$ to \eqref{eq:ipm} and $T,\delta>0$  fulfilling, for $t\in [0,T]$
    \begin{enumerate}
        \item $\rho_{pert}(x,t)=\rho (x,t)+x_2$, with $\|\rho_{pert}(x,0)\|_{H^2}\leq \epsilon_{0}$,
        \item $\rho_{pert}(x,t)=\rho (x,t)+x_2 \in C^\infty (\R^2)$  and $\supp (\rho(x,t)) \cap   B_\delta(0)=\emptyset$, 
        \item $\partial_{x_1}u_{1}[\rho](x,t)=\partial_{x_1}u_{1}[\rho_{pert}](x,t)<0$.
    \end{enumerate}
\end{lem}
\begin{Rmk}
As $\rho_{pert}$ satisfies \eqref{eq:ipm-stable}, notice the validity of the estimate 
\begin{align*}
        \frac{d}{dt}\|\rho_{pert}\|_{H^4}^2 \lesssim \|\rho_{pert}\|_{H^4}^3,
    \end{align*}
which implies that $\|\rho_{pert}(t)\|_{H^4}$ is uniformly bounded for all $t \in [0, T_{pert}]$ with $T_{pert} \sim \|\rho_{pert}(x, 0)\|_{H^4}^{-1}$. This follows from the well-posedness of the stable equation \eqref{eq:ipm-stable} in $H^{2+\epsilon}(\R^2)$ for any $\epsilon>0$, as mentioned in the introduction.
\end{Rmk}
\begin{proof}
    We will first use Lemma \ref{initialdata} to find $\rho_{\text{in}}(x)$ such that $\|\rho_{\text{in}}(x)\|_{H^2}\leq \frac{\epsilon_{0}}{2}$ with $\rho_{\text{in}}(x)=x_{2}$ for $x$ in some small ball $B_{\delta_0}(0)$. Note that, in particular, $-x_{2}+\rho_{\text{in}}(x)=0$ for $x\in B_{\delta_0}(0)$, and also that $|\partial_{x_1}u_{1}[\rho_{\text{in}}](x=0)|< \infty$.
    Furthermore, we consider
    $$\rho_{K,\delta_0}(x)=\delta_0\sum_{i=1}^{K}\frac{f(2^{i}x)}{i2^{i}}$$
    where $f(x)\in C^{\infty}$ is a symmetric function such that ${x_{2}f(x)<0}$, with
    $$\supp f(x)\subset (B_{2}(0)\setminus B_{1}(0))\cap \{(x_{1},x_{2}): |x_{2}|\geq \sqrt{3} |x_{1}|\}.$$ 
    Using the fact that $f(2^{i_{1}}x)$ and $f(2^{i_{2}}x)$ have disjoint support if $i_{1}\neq i_{2}$, we have that
    $$\|\rho_{K,\delta_0}\|^2_{H^2}=\delta_0^2\sum_{i=1}^{K}\left\|\frac{f(2^{i}x)}{i2^{i}}\right\|^2_{H^2}\leq \delta_0^2 \|f(x)\|^2_{H^2}\sum_{i=1}^{K}\frac{1}{i^2}\leq C\delta_0^2$$
    and therefore, by taking $\delta_0$ small we obtain
    $$\|f_{k,\delta_0}\|_{H^2}\leq \frac{\epsilon_{0}}{2}.$$
    Note that, relying on the explicit formula
    \begin{equation}\label{eq:der-u1}
    \de_{x_1} u_1 [\rho](x=0, t) = \frac{1}{\pi} \int_{\R^2}  \frac{y_2 (y_2^2-3y_1^2)}{|\by|^6} \rho (t, y_1, y_2) \, d y_1 \, d y_2,
    \end{equation}
    and using that $\supp f(x) \subset \{(x_{1},x_{2}): |x_{2}|\geq \sqrt{3} |x_{1}|\}$ and $x_2 f(x) < 0$, we deduce
    $$\partial_{x_1}u_{1}\left[\frac{f(2^{i}x)}{2^{i}}\right](x=0)=\partial_{x_1}u_{1}[f(x)](x=0)=\frac{1}{\pi} \int_{\R^2}  \frac{y_2 (y_2^2-3y_1^2)}{|\by|^6} f(y_1, y_2) \, d y_1 \, d y_2<0,$$
    which means that, for any positive number $N>0$, for $K$ big enough, we have
    $$\partial_{x_1}u_{1}[\rho_{K,\delta_0}](x=0)=\delta_0 \partial_{x_1}u_{1}[f(x)](x=0)\sum_{i=1}^{K}\frac{1}{i}\leq  -N.$$
    In particular, by taking $K$ big we achieve
    $$\partial_{x_1}u_{1}[\rho_{\text{in}}+\rho_{K,\delta_0}](x=0)<0.$$
    We will then consider $\rho(x,t)$ the solution to \eqref{eq:ipm} with initial conditions $-x_{2}+\rho_{\text{in}}+\rho_{K,\delta_0}$. Note that, by the local well-posedness of the \eqref{eq:ipm} equation around the stable profile $-x_{2}$, we know that $\rho(x,t)\in C^{\infty}$ exists for some short period of time $[0,T_{1}]$. Furthermore, by the choice of the initial condition, our lemma is already fulfilled at the initial time. 

    To prove that there exists $\delta>0$ such that $\supp (\rho(x,t)) \cap   B_\delta=\emptyset$, we use that, for $t=0$, both $-x_{2}+\rho_{\text{in}}(x)$ and $\rho_{K,\delta_0}(x)$ have support away from the origin. Furthermore, the symmetry of $\rho(x,t)$ means that that $\mathbf{u}[\rho](0,t)=0$, and $\mathbf{u}[\rho](x,t)\in C^1$. This implies that the support of $\rho(x,t)$ can only approach the origin exponentially fast (since the support moves with the velocity $\mathbf{u}[\rho](x,t)$ and $|\mathbf{u}(x,t)|\leq C|x|$), and thus $\rho(x,t)$ is also supported away from the origin for $t\in[0,T_{1}]$.

    Finally, to show that $\partial_{x_1}u_{1}[\rho](x=0,t)<0$, we note that $\partial_{x_1}u_{1}[\rho](x=0,0)<0$, and using the well-posedness of the solution in $C^{\infty}$, which in particular gives us continuity of the solution, provides $\partial_{x_1}u_{1}[\rho](x=0,t)<0$ for $t\in[0,T_{2}]$. Taking $T=\text{min}(T_{1},T_{2})$ concludes the proof.
\end{proof}

To achieve a solution inducing a pronounced deformation around the origin, we employ an iterative process. Assuming we have a solution $\rho_{j}(x,t)$ such that 
\[
\int_{0}^{T}\partial_{x_{1}} u_1[\rho_{j}](x=0,t)\, dt = -M,
\]
with $M > 0$,
we would like to modify this solution by adding a perturbation $\rho_{\text{pert}}(x,0)$. Considering the solution to \eqref{eq:ipm} as $\rho_{j+1}(x,t)$ with initial condition $\rho_{j}(x,0) + \rho_{\text{pert}}(x,0)$, we aim for $\rho_{j+1}(x,t)$ to satisfy 
\[
\int_{0}^{T}\partial_{x_{1}} u_1[\rho_{j+1}](x=0,t)\, dt \leq -(M + c_{M}),
\]
where $c_{M} > 0$. Before doing that, we need to prove the stability of the deformation at the origin over time, in the intermediate result below.

\subsection{Persistence of the deformation at the origin}
\begin{lem} \label{stabledeformation}
For any $M,T>0$, suppose that $k(t)$ satisfies 
\begin{align*}
    k(t)<0, \qquad  0< -\int_0^T k(t) \, d t \le M,
\end{align*}
for all $t \in [0,T]$. 
There exists a constant $\mathfrak C >0$, such that the following holds.
Given a symmetric function $\tilde{\rho}_{\text{in}}(x)$ with $x_{2}\tilde{\rho}_{\text{in}}(x)\leq 0$, and
\begin{align}\label{hp:supp}
\supp (\tilde \rho_{\text{in}}(x))\subset\mathcal{D}:=\{(x_1, x_2) \, | \, |x_{2}| > \mathfrak C |x_{1}|\},
\end{align}
suppose that $\tilde \rho (x, t)$ solves
\begin{equation}\label{eq:lemmastabledef}
\begin{cases}
    \de_t \tilde{\rho} + (k(t)(x_{1},-x_{2})) \cdot \nabla \tilde{\rho}=0, \quad t \in [0,T]\\
    \tilde \rho (x, 0)=\tilde \rho_{\text{in}}(x).
\end{cases}
\end{equation}
Then $\tilde{\rho}(x,t)$ is symmetric and, for $t\in[0,T]$, it satisfies:
\begin{align}
    \de_{t}\de_{x_1} u_1[\tilde{\rho}](x=0,t)&\geq 0; \quad (\text{and} \quad 
    \de_{x_1} u_1[\tilde{\rho}](x=0,t)\geq \de_{x_1} u_1 [\tilde{\rho}](x=0, t=0)); \\
    \de_{x_1} u_1[\tilde{\rho}](x=0,t)&\leq  {\text{e}^{-7M}}\de_{x_1} u_1[\tilde{\rho}](x=0, t=0). \label{ineq:lemma1.3-lower}
\end{align}
\end{lem}
\begin{proof}
First, recall the explicit formula in \eqref{eq:der-u1}:
\begin{equation*}
\de_{x_1} u_1 [\tilde \rho](x=0, t) = \frac{1}{\pi} \int_{\R^2}  \frac{y_2 (y_2^2-3y_1^2)}{|\by|^6}\tilde \rho(t, y_1, y_2) \, d y_1 \, d y_2.
\end{equation*}
In polar coordinates $(y_1', y_2')=r'  (\cos \alpha', \sin \alpha')$, it reads
\begin{equation}\label{eq:der-u1-polar}
\de_{x_1} u_1 [\tilde \rho](x=0, t) = \frac{1}{\pi}  \int_0^{2\pi} \int_0^\infty   -\frac{\sin(3\alpha ')}{(r')^2} \tilde \rho(t, r', \alpha') \, d r' \, d \alpha'.
\end{equation}
Notice that if $\tilde{\rho}_{\text{in}}(x)x_{2}<0$, then using \eqref{eq:der-u1}
\begin{align*}
    \de_{x_1}u_1[\tilde \rho](x=0, t=0)<0 \quad \text{for} \quad \text{supp}(\tilde{\rho}_{\text{in}}(x)) \in \mathcal{D} \quad \text{with} \quad \mathfrak C > \tan \frac{\pi}{3}=\sqrt 3. 
\end{align*}
We need to estimate $\de_{x_1}u_1[\tilde \rho](x=0, t)$. 
Let us consider the flow map $\Phi (x_1, x_2, t)=(\phi_1(x_1, t), \phi_2(x_2, t))^T$ associated with the transport equation \eqref{eq:lemmastabledef}. Explicitly, we have that
\be
\Phi (x_1, x_2, t)=\begin{pmatrix}
\phi_1(x_1, t)\\
\phi_2(x_2, t)
\end{pmatrix}=  \begin{pmatrix}
    x_1 & 0 \\
    0 & x_2
\end{pmatrix}\begin{pmatrix} \mathrm{e}^{\int_0^t k (\tau) \, d\tau} \\
\mathrm{e}^{-\int_0^t k (\tau) \, d\tau}
\end{pmatrix}.
\ee
Inserting the above into \eqref{eq:der-u1} and relying on the change of coordinates induced by the flow map yields 
\begin{align}
    \de_{x_1} u_1 [\tilde \rho](x=0, t) & = \frac{1}{\pi} \int_{\R^2}  \frac{y_2 (y_2^2-3y_1^2)}{|\by|^6}\tilde \rho_{\text{in}}\left( \phi_1^{-1}(y_1, t), \phi_2^{-1}(y_2, t) \right) \, d y_1 \, d y_2\notag\\
    & = \frac{1}{\pi} \int_{\R^2}  \frac{\phi_2 (\phi_2^2-3\phi_1^2)}{|\Phi|^6}\tilde \rho_{\text{in}}( y_1, y_2) \, d y_1 \, d y_2.
\end{align}
First, we would like to show that $\de_{x_1} u_1 [\tilde \rho](x=0, t)<0$ for all $t\in [0,T]$. 
To ensure that, since $\tilde \rho_{\text{in}} \phi_2 < 0$, we want
$\phi_2^2-3\phi_1^2 >0$, which amounts at requiring
\begin{align}
    |y_2|^2 > 3 \mathrm{e}^{4\int_0^t k(\tau) \, d \tau}|y_1|^2,
\end{align}
and this is satisfied as soon as we choose $\mathfrak C \geq \sqrt{3}$ in \eqref{hp:supp} as $k(t)<0$.
\\
Now, upon relying on the change of coordinates induced by the flow map, consider
\begin{align}
    \de_t \de_{x_1} u_1 [\tilde \rho](x=0, t)&= \frac{1}{\pi}\int_{\R^2_+} \de_t \left(\frac{\phi_2(y_2, t) (\phi_2^2(y_2, t)-3\phi_1^2(y_1, t))}{|\Phi (y_1, y_2, t)|^6}\right)\tilde \rho_{\text{in}}(y_1, y_2) \, d y_1 \, d y_2\notag\\
    &= \frac{k(t)}{\pi} \int_{\R^2_+}\frac{\phi_2(y_2, t) (3\phi_2^4(y_2, t)+15\phi_1^4(y_1, t)-30 \phi_1^2(y_1, t) \phi_2^2(y_2, t))}{|\Phi (y_1, y_2, t)|^8} \tilde \rho_{\text{in}}(y_1, y_2) \, d y_1 \, d y_2.
\end{align}
Since $k(t)<0$ and $\phi_2(y_2, \cdot) \tilde \rho_{\text{in}}(y_1, y_2)<0$ for $(y_1, y_2) \in \supp (\tilde \rho_{\text{in}})$, we would like to prove that 
\begin{align}\label{ineq:30}
    30 \phi_1^2(y_1, t) \phi_2^2(y_2, t)  \le 3\phi_2^4(y_2, t)+15\phi_1^4(y_1, t),
\end{align}
namely 
\begin{align*}
    30 y_1^2 y_2^2 \le 3 \mathrm{e}^{-4\int_0^t k(\tau) \, d\tau} y_2^4 + 15\mathrm{e}^{4\int_0^t k(\tau) \, d\tau} y_1^4.
\end{align*}
This is true as soon as  
\begin{align*}
   30 y_1^2 y_2^2 \le 3 \mathrm{e}^{-4\int_0^t k(\tau) \, d\tau} y_2^4,
\end{align*}
which holds provided that
\begin{align*}
    |y_2| \geq\sqrt{10}\mathrm{e}^{2\int_0^t k(\tau) \, d \tau} |y_1|.
\end{align*}
In the definition of the domain \eqref{hp:supp}, we then require
\begin{align*}
    \mathfrak C \geq \sqrt{10}.
\end{align*}
It remains to show the last inequality:
\begin{align*}
    \de_{x_1}u_1[\tilde \rho](x=0, t) & = \frac{1}{\pi}\int_{\R^2_+} \frac{\mathrm{e}^{-\int_0^t k(\tau) \, d \tau} y_2 \left(\mathrm{e}^{-2\int_0^t k(\tau) \, d \tau}y_2^2 - 3 \mathrm{e}^{2\int_0^t k(\tau) \, d \tau} y_1^2\right)}{\left(\mathrm{e}^{2\int_0^t k(\tau) \, d \tau}y_1^2 + \mathrm{e}^{-2\int_0^t k(\tau) \, d \tau}y_2^2 \right)^3}\, \tilde \rho_{\text{in}} ( y_1, y_2)\,  d y_1 \, d y_2\\
    & \leq \frac{1}{\pi} \int_{\R^2_+} \frac{\mathrm{e}^{-\int_0^t k(\tau) \, d \tau} y_2 \left(\mathrm{e}^{-2\int_0^t k(\tau) \, d \tau}y_2^2 - 3 \mathrm{e}^{2\int_0^t k(\tau) \, d \tau} y_1^2\right)}{\left(\mathrm{e}^{-2\int_0^t k(\tau) \, d \tau} |y|^2 \right)^3}\,  \tilde \rho_{\text{in}} (y_1, y_2) \, d y_1 \, d y_2 \\
    & = \frac{\mathrm{e}^{7\int_0^t k(\tau) \, d \tau}}{\pi} \int_{\R_+^2} \frac{y_2}{|y|^6} \left(\mathrm{e}^{-4\int_0^t k(\tau) \, d \tau}y_2^2-3 y_1^2\right)\,  \tilde \rho_{\text{in}} (y_1, y_2) \, d y_1 \, d y_2\\
    & = \frac{\mathrm{e}^{7\int_0^t k(\tau) \, d \tau}}{\pi} \int_{\R_+^2} \frac{y_2}{|y|^6} \left(y_2^2-3 y_1^2\right)\,  \tilde \rho_{\text{in}} (y_1, y_2) \, d y_1 \, d y_2\\
    &\quad + \frac{3\mathrm{e}^{7\int_0^t k(\tau) \, d \tau}}{\pi} \int_{\R_+^2} \frac{y_2}{|y|^6} \left(\mathrm{e}^{-4\int_0^t k(\tau) \, d \tau}-1\right)y_2^2 \, \tilde \rho_{\text{in}} (y_1, y_2) \, d y_1 \, d y_2\\
    & \leq {\mathrm{e}^{-7M}}\de_{x_1}u_1[\tilde \rho](x=0, t=0),
\end{align*}
where the last inequality holds since  $\left(\mathrm{e}^{-4\int_0^t k(\tau) \, d \tau}-1\right)>0$ (as $k(t)<0$) and $y_2 \tilde \rho_{\text{in}}<0$.
\end{proof}
Lemma \ref{stabledeformation} allows us to derive useful properties of the solutions we construct; however, it is applicable only when $k(t) < 0$. Specifically, we will apply this lemma in the case where $\partial_{x_{1}} u_{1} < 0$, which is precisely the setting of the preceding Lemma \ref{lem:stable-exact-sol}.

\section{Gluing interior and exterior approximate solutions}
Now, after proving the stability of the deformation at the origin over time in Lemma \ref{stabledeformation}, we approximate the solution to our equations with different layers and for that we need a gluing lemma.
%
%
%
\begin{lem}\label{gluing1}
    Let $T>0, M>0$ be given. For $t \in [0, T]$, let $\rho(x,t) \in  C([0,T]; H^4 (\R^2))$ be a symmetric solution, in the sense of Definition \ref{def:symmetric}, to the IPM equation \eqref{eq:ipm} with stable initial data \eqref{eq:data-stable}. Furthermore, assume that $\supp(\rho(x,t)) \cap B_\delta (0) = \emptyset $ for some $\delta>0$, and let $f(x)$ be a compactly supported $H^4(\R^2)$ function. Setting
    \begin{align*}
        k(t)&=(\de_{x_{1}}u_1[\rho])(x=0,t)<0, \quad -\int_{0}^{T}k(t) \, d t \leq M,
    \end{align*}
    introduce
    \begin{align}\label{simp1}
    &\begin{cases}
        \partial_{t}\bar{\rho}_{\lambda, \text{inter}}(x,t)+\mathbf{u}[\bar{\rho}_{\lambda, \text{inter}}]\cdot \nabla \bar{\rho}_{\lambda, \text{inter}}+k(t)(x_{1},-x_{2})\cdot\nabla \bar{\rho}_{\lambda, \text{inter}}=0,\\
        \bar{\rho}_{\lambda, \text{inter}}(x,0)=\frac{f(\lambda x)}{\lambda},
    \end{cases}\\
    \label{eq:sum1}
    &\begin{cases}
         \partial_{t}\tilde{\rho}_{\text{ext}}(x,t)+\mathbf{u}[\tilde{\rho}_{\lambda, \text{inter}}+\tilde{\rho}_{\text{ext}}]\cdot \nabla \tilde{\rho}_{\text{ext}}=0, \\
    \tilde{\rho}_{\text{ext}}(x,0)=\rho(x,0)=-x_2+\rho_{\text{in}}(x),
    \end{cases}\\
    &\begin{cases}
    \partial_{t}\tilde{\rho}_{\lambda, \text{inter}}(x,t)+\bold{u}[\tilde{\rho}_{\lambda, \text{inter}}+\tilde{\rho}_{\text{ext}}]\cdot \nabla \tilde{\rho}_{\lambda, \text{inter}}=0,\notag\\
    \tilde{\rho}_{\lambda, \text{inter}}(x,0)=\frac{f(\lambda x)}{\lambda}.
    \end{cases} 
    \end{align}
    Assume that $\|\rho(x,t)+x_2\|_{H^4}$ and $\left\|\lambda \bar{\rho}_{\lambda, \text{inter}}\left(\frac{x}{\lambda}, t\right)\right\|_{H^4}$ are uniformly bounded for $t \in [0,T]$, $\lambda>1$.
    Then, there exists a constant $C>0$, depending on $\delta>0$, $T$, $M$, $\|\rho(x,t)+x_2\|_{H^4}$, $\|\lambda \bar{\rho}_{\lambda, \text{inter}}(\frac{x}{\lambda}, t)\|_{H^4}$,
    such that, for $\lambda>0$ large enough, the following estimates hold:
    \begin{align}
    &\|\rho(x, t)-\tilde{\rho}_{\text{ext}}(x,t)\|_{H^3}\le C t \lambda^{-1}, \quad \de_{t}\|\rho(x,t)-\tilde{\rho}_{\text{ext}}(x,t)\|_{H^3}\le C \lambda^{-1},\label{est:nonscaled-new}\\
    & \left\|\lambda \tilde{\rho}_{\lambda, \text{inter}}\left(\frac{x}{\lambda}\right)-\lambda \bar{\rho}_{\lambda, \text{inter}}\left(\frac{x}{\lambda}\right)\right\|_{H^3}\leq C t\lambda^{-1},\de_{t}\left\|\lambda \tilde{\rho}_{\lambda, \text{inter}}\left(\frac{x}{\lambda}\right)-\lambda \bar{\rho}_{\lambda, \text{inter}}\left(\frac{x}{\lambda}\right)\right\|_{H^3}\leq C\lambda^{-1}.\label{est:scaled-new}
    \end{align}
    Finally, $\|\tilde \rho_{\text{ext}}(x, t)+x_2\|_{H^4}$ and $\left\|\lambda \tilde{\rho}_{\lambda, \text{inter}}\left(\frac{x}{\lambda}\right)\right\|_{H^4}$ are uniformly bounded with respect to $\lambda$ for all $t \in [0,T]$.
\end{lem}
\begin{Rmk}
    Notice that $\tilde{\rho}_{\text{ext}}(x,t)+\tilde{\rho}_{\lambda, \text{inter}}(x,t)$ is simply the solution to \eqref{eq:ipm} with initial conditions $\rho(x,0)+\tilde \rho_{\lambda, \text{inter}}(x,0)$. On the other hand, $\rho(x,t)+\bar{\rho}_{\lambda, \text{inter}}(x,t)$ represents our naive approximation. Here, we demonstrate that under certain circumstances, we can significantly reduce the error of our approximation (in appropriate spaces).
\end{Rmk}
\begin{Rmk}
To ensure that we can find a solution $\rho(x,t)$ with the properties required in the statement of Lemma \ref{gluing1}, we apply Lemma \ref{initialdata} yielding an initial datum $\rho_{\text{in}}(x)$ such that, for some $\delta_0>0$,
\begin{align*}
    \supp (\rho (x, 0))=\supp (-x_2+\rho_{\text{in}}(x)) \subset \R^2\setminus B_{\delta_0}.
\end{align*}
Notice that such $\rho_{\text{in}}(x)$ is also involved in the initial datum of $\tilde \rho_{\text{ext}}(x, t)$.
Next, Lemma \ref{lem:stable-exact-sol} provides an exact solution $\rho(x,t)$ to \eqref{eq:ipm} with initial data as above, such that 
\begin{align*}
    \supp (\rho (x, t)) \subset \R^2\setminus B_{\delta}(0), \quad \delta \le \delta_0.
\end{align*}
\end{Rmk}
\begin{proof}[Proof of Lemma \ref{gluing1}]
By continuity in time, which is guaranteed by the well-posedness of the IPM equation (in this case in $H^3$), since $\rho (x, t)$ and $\tilde \rho_{\text{ext}} (x, t)$ depart from  the same initial data as well as $\tilde{\rho}_{\lambda, \text{inter}}$ and $\bar{\rho}_{\lambda, \text{inter}}$, for any $\epsilon>0$ there exists $T_\epsilon>0$ such that 
$$\|\rho(x,t)-\tilde{\rho}_{\text{ext}}(x,t)\|_{H^3}\le \epsilon,\quad \left\|\lambda \tilde{\rho}_{\lambda, \text{inter}}\left(\frac{x}{\lambda}, t\right)-\lambda \bar{\rho}_{\lambda, \text{inter}}\left(\frac{x}{\lambda}, t\right)\right\|_{H^3}\leq \epsilon, \quad t \in [0, T_\epsilon].$$
By using the hypotheses, this gives
\begin{align}\label{ineq:shortime}
\|\tilde{\rho}_{\text{ext}}(x,t)+x_2\|_{H^3}\le C+\epsilon,\quad \left\|\lambda \tilde{\rho}_{\lambda, \text{inter}}\left(\frac{x}{\lambda}, t\right)\right\|_{H^3}\leq C +\epsilon,
\end{align}
for some constant $C>0$.
We will first show that, during the time when this holds, we can control the support of $\tilde{\rho}_{\lambda, \text{inter}},\tilde{\rho}_{\text{ext}}$.
\subsubsection{Control of the support of $\tilde\rho_{\lambda, \text{inter}}, \tilde \rho_{\text{ext}}$}\label{sec:support}
Let us introduce the support size
\begin{align}
    \mathcal{S}_\lambda^{inf} (t):=\inf \{ \mu \in \mathbb{R}_+\, : \, \tilde \rho_{\lambda, \text{inter}} (t, x_1, x_2)=0 \; \text{for all} \; |x_{1}|^2+|x_{2}|^2\ge \mu^2\}. 
\end{align}
Denoting by $\mu_0 \in \R_+$ the size of the support of $f(x)$, at the initial time we have
\begin{align}
    \mathcal{S}_\lambda^{inf}(0)=\frac{\mu_0}{\lambda}.
\end{align}
Let $\tilde \Phi^\lambda (x, t)=(\tilde \phi_1^\lambda (x, t), \tilde{\phi}_2^\lambda (x, t))$ be the flow map associated with the equation for $\tilde \rho_{\lambda, \text{inter}}$. We have
\begin{align}
    {\tilde \Phi}^\lambda (x, t)=  x + \int_0^t \mathbf{u}[\tilde \rho_{\lambda, \text{inter}} + \tilde \rho_{\text{ext}}](\tilde\Phi^\lambda (x, s), s) \, d s.
\end{align}
Then, by the Cauchy-Lipschitz Theorem, and since $u[\tilde \rho_{\lambda, \text{inter}} + \tilde \rho_{\text{ext}}](x=0)=0$,
\begin{align}\label{eq:support-tildelambda}
  \frac{\mu_0}{\lambda} \text{e}^{-t\| \mathbf{u}[\tilde \rho_{\lambda, \text{inter}} + \tilde \rho_{\text{ext}}]\|_{C^1}} \le \mathcal{S}_\lambda^{inf}(t) \le  \frac{\mu_0}{\lambda} \text{e}^{t\| \mathbf{u}[\tilde \rho_{\lambda, \text{inter}} + \tilde \rho_{\text{ext}}]\|_{C^1}}.
\end{align}
Using that, for $\lambda\geq 1$
$$\|f(x, t)\|_{C^1} \leq \left\|\lambda f \left( \frac{x}{\lambda}, t\right)\right\|_{C^1},$$
we can bound
\begin{align*}
    \|\mathbf{u}[\tilde \rho_{\lambda, \text{inter}} + \tilde \rho_\text{ext}](x, t) \|_{C^1} & \le  \|\mathbf{u}[\tilde \rho_{\lambda, \text{inter}}](x, t) \|_{C^1}+ \|\mathbf{u}[\tilde \rho_{\text{ext}}](x, t)\|_{C^1} \\
    &\le  \left\|\lambda\mathbf{u}[\tilde \rho_{\lambda, \text{inter}}]\left(\frac{x}{\lambda}, t\right) \right\|_{C^1}+ \|\mathbf{u}[\tilde \rho_{\text{ext}}](x, t)\|_{C^1} \\
    & \le \left\|\lambda\tilde \rho_{\lambda, \text{inter}}\left(\frac{x}{\lambda}, t\right) \right\|_{H^3}+\|\mathbf{u}[\tilde \rho_{\text{ext}}+x_2](x, t)+ \mathbf{u}[-x_2]\|_{C^1}
    \le C ,
\end{align*}
where in the latter we used that $\mathbf{u}[-x_2]=0$, the embedding $H^3(\mathbb R^2) \hookrightarrow C^1 (\mathbb R^2)$ and \eqref{ineq:shortime}. From \eqref{eq:support-tildelambda}, this implies that
\begin{align}\label{eq:support-tildelambda-1}
     \frac{\mu_0}{\lambda} \text{e}^{-t(C+\eps)} \le \mathcal{S}_\lambda^{inf}(t) \le  \frac{\mu_0}{\lambda}\text{e}^{t(C+\eps)},
\end{align}
where $\mu_0$ is fixed, while $\lambda$ can be chosen big enough.
Similarly, 
introducing
\begin{align}\label{eq:supp-sup}
    \mathcal{S}^{sup}(t):= \sup \{ \mu \in \R_+ \, : \, \tilde \rho_{\text{ext}} (t, x_1, x_2)=0 \; \text{for all} \; |x_1|^2+|x_2|^2\le \mu^2\},
\end{align}
at the initial time, by Lemma \ref{initialdata}, we have
\begin{align*}
    \mathcal{S}^{sup}(0) = \delta_0.
\end{align*}
Again, by the Cauchy-Lipschitz Theorem, this is is bounded at later times by 
\begin{align*}
     \mathcal{S}^{sup}(t) \ge \delta_0 \text{e}^{- (C+\eps) t}.
\end{align*}
We would like to ensure separation of the supports of $\tilde \rho_{\text{ext}} (x,t)$ and $\tilde \rho_{\lambda, \text{inter}} (x, t)$ for all times $t \in [0,T]$. This amount at choosing $\lambda$ big enough to satisfy
\begin{align}
    \delta_0 \text{e}^{- (C+\eps) t} > \frac{\mu_0}{\lambda} \text{e}^{ (C+\eps) t}, \quad t \in [0, T].
\end{align}
This can be easily arranged by taking $\lambda$ big.
\subsubsection{Estimating $\|\tilde \rho_{\text{ext}} - \rho\|_{H^3}$.}
In the course of the proof, we drop the subscripts ``inter, ext''. 
Consider now the equation for the difference
\begin{align}
    \de_t (\tilde \rho - \rho)+\mathbf{u}[\tilde \rho_\lambda + \tilde \rho] \cdot \nabla (\tilde \rho - \rho) + \mathbf{u}[\tilde \rho_\lambda + \tilde \rho-\rho] \cdot \nabla \rho = 0.
\end{align}
Multiplying by $(\tilde \rho - \rho)$, we obtain
\begin{align}
    \frac 12 \frac{d}{d t}\|\tilde \rho - \rho\|^2_{L^2} + (\mathbf{u}[\tilde \rho_\lambda + \tilde \rho-\rho] \cdot \nabla \rho, \tilde \rho - \rho)_{L^2}=0.
\end{align}
Differentiating in space, we get
\begin{align}
    &\de_t D^j (\tilde \rho - \rho) + \mathbf{u}[\tilde \rho_\lambda + \tilde \rho] \cdot \nabla D^j (\tilde \rho - \rho) + [D^j, \bold \mathbf{u}[\tilde \rho_\lambda + \tilde \rho] ] \cdot \nabla (\tilde \rho - \rho)
    +\mathbf{u}[\tilde \rho_\lambda + \tilde \rho-\rho] \cdot \nabla D^j \rho\notag\\ &\quad + [D^j, \mathbf{u}[\tilde \rho_\lambda + \tilde \rho-\rho]]\cdot \nabla \rho=0. 
\end{align}
We want to get an estimate of the difference in $H^3$ (notice that if $j=3$, then the second-to-last addend above contains the term $\nabla D^3 \rho$, whose control needs $D^4\rho$ in $L^2$).
Taking the scalar product against $D^j (\tilde \rho - \rho)$,
\begin{align*}
\frac{d}{d t}\|D^j (\tilde \rho - \rho)\|_{L^2}^2 &\lesssim \|[D^j, \mathbf{u}[\tilde \rho_\lambda + \tilde \rho-\rho]]\cdot \nabla (\tilde \rho - \rho)\|_{L^2} \| D^j(\tilde \rho - \rho)\|_{L^2}\\
&\quad +\|[D^j, \mathbf{u}[\rho]]\cdot \nabla (\tilde \rho - \rho)\|_{L^2} \| D^j(\tilde \rho - \rho)\|_{L^2}\\
&\quad +  \|[D^j, \mathbf{u}[\tilde \rho_\lambda + \tilde \rho-\rho]]\cdot \nabla  \rho\|_{L^2} \| D^j(\tilde \rho - \rho)\|_{L^2}\\
& \quad + |(\mathbf{u}[\tilde \rho_\lambda + \tilde \rho-\rho] \cdot \nabla D^j \rho, D^j(\tilde \rho - \rho))_{L^2}|\\
& \lesssim \|[D^j, \mathbf{u}[\tilde \rho_\lambda]]\cdot \nabla (\tilde \rho - \rho)\|_{L^2} \| D^j(\tilde \rho - \rho)\|_{L^2}+\|[D^j, \mathbf{u}[\tilde \rho-\rho]]\cdot \nabla (\tilde \rho - \rho)\|_{L^2} \| D^j(\tilde \rho - \rho)\|_{L^2}\\
&\quad +\|[D^j, \mathbf{u}[\rho]]\cdot \nabla (\tilde \rho - \rho)\|_{L^2} \| D^j(\tilde \rho - \rho)\|_{L^2}\\
&\quad +  |([D^j, \mathbf{u}[ \tilde\rho_\lambda]]\cdot \nabla  \rho, D^j(\tilde \rho - \rho))_{L^2}| +  \|[D^j, \mathbf{u}[\tilde \rho-\rho]]\cdot \nabla  \rho\|_{L^2} \| D^j(\tilde \rho - \rho)\|_{L^2}\\
& \quad + |(\mathbf{u}[\tilde \rho_\lambda] \cdot \nabla D^j \rho, D^j(\tilde \rho - \rho))_{L^2}|+ |(\mathbf{u}[\tilde \rho-\rho] \cdot \nabla D^j \rho, D^j(\tilde \rho - \rho))_{L^2}|.
\end{align*}
To bound the third-to-last line, we use that
\begin{align*}
    \mathbf{u}[\rho]=\mathbf{u}[\rho+x_2].
\end{align*}
Relying on the boundedness of (composition of) Riesz Transform in $L^2$-based spaces and recalling that $\rho + x_2 \in H^4$ yields 
\begin{align}
\frac{d}{d t}\|D^j (\tilde \rho - \rho)\|_{L^2}^2 
& \lesssim (\|\rho+x_2\|_{H^3}+\|\tilde \rho-\rho\|_{H^3}+\|\tilde \rho_\lambda\|_{H^3})\|\tilde \rho-\rho\|_{H^3}^2 +  |([D^j, \mathbf{u}[ \tilde \rho_\lambda]]\cdot \nabla  \rho, D^j(\tilde \rho - \rho))|\notag\\
&\quad + |(\mathbf{u}[\tilde \rho_\lambda] \cdot \nabla D^j \rho, D^j(\tilde \rho - \rho))_{L^2}| + |(\mathbf{u}[\tilde \rho-\rho] \cdot \nabla D^j \rho, D^j(\tilde \rho - \rho))_{L^2}|.
\label{eq:effective}
\end{align}
Among the remaining terms above, there is in particular
\begin{align*}
    ([D^j, \mathbf{u}[ \tilde \rho_\lambda]]\cdot \nabla  \rho, D^j(\tilde \rho - \rho)).
\end{align*}
We notice, thanks to Leibniz rule and the $\lambda$-scaling, that the potentially most dangerous among the above terms is 
\begin{align}\label{eq:addends-new}
    |(D^j \bold u [\tilde \rho_\lambda] \cdot \nabla \rho, D^j (\tilde \rho-\rho))_{L^2}|
    \lesssim \| D^j u_1[\tilde \rho_\lambda] \de_{x_1} \rho\|_{L^2} \| \tilde \rho - \rho\|_{H^3}+\| D^j u_2[\tilde \rho_\lambda] \de_{x_2} \rho\|_{L^2} \| \tilde \rho - \rho\|_{H^3}.
\end{align}
We treat the two terms above separately.
We recall that by assumption, the support $\text{supp}(\rho (x, \cdot)) \cap B_\delta (0)=\emptyset$. Moreover, we know from Section \ref{sec:support} that $\text{supp}(\tilde\rho_\lambda (x, \cdot)) \cap (\R^2\setminus B_{c\lambda^{-1}} (0))=\emptyset$ for some constant $c>0$. These observations imply that we can simply estimate
\begin{align}\label{eq:estu-supp-new}
    \| D^j u_1[\tilde \rho_\lambda]\|_{L^2(\R^2\setminus B_\delta(0))} & = \|D^j H_1 \star \tilde \rho_\lambda\|_{L^2(\R^2\setminus B_\delta(0))}  \lesssim \left\| \int_{B_{c\lambda^{-1}}(0)}\frac{1}{|x-y|^{2+j}} |\tilde \rho_\lambda (y, \cdot)| \, d y_1 \, d y_2 \right\|_{L^2(\R^2\setminus B_\delta(0))},
\end{align}
where we noticed that the kernel $H_1$ in \eqref{eq:kernels} satisfies 
\begin{align}
    |D^j H_1 (x-y)| \lesssim \frac{1}{|x-y|^{2+j}}.
\end{align}
By the elementary inequality
\begin{align}
    |x-y| \ge \|x|-|y\| \ge \delta - c\lambda^{-1} \ge \frac{\delta}{2}
\end{align}
for any given $\delta>0$ and for all $\lambda>0$ large enough, and using Young convolution inequality
\begin{align*}
    \| D^j u_1[\tilde \rho_\lambda]\|_{L^2(\R^2\setminus B_\delta(0))}&  \lesssim C(\delta) \|\tilde \rho_\lambda\|_{L^2}.
\end{align*}
This gives the estimate
\begin{align}
\| D^j u_1[\tilde \rho_\lambda] \de_{x_1} \rho\|_{L^2(\R^2\setminus B_\delta (0))} & \lesssim \| D^j u_1[\tilde \rho_\lambda]\|_{L^2(\R^2\setminus B_\delta (0))} \|\de_{x_1}\rho\|_{L^\infty}\lesssim C(\delta) \|\tilde \rho_\lambda\|_{L^2}\|\nabla\rho\|_{L^\infty}
\end{align}
and, recalling that the initial data satisfy $\|\tilde \rho_\lambda (x, 0)\|_{L^2} = \lambda^{-1} \| f(\lambda x)\|_{L^2} \lesssim \lambda^{-2}$, thanks to the conservation of the $L^2$ norm of $\tilde \rho_\lambda$ we have
\begin{align}
    \|\tilde \rho_\lambda\|_{L^2} \lesssim \lambda^{-2},
\end{align}
yielding the following estimate
\begin{align}
     \| D^j u_1[\tilde \rho_\lambda] \de_{x_1} \rho\|_{L^2} \|\tilde \rho-\rho\|_{L^2} & \lesssim C(\delta) \|\tilde \rho-\rho\|_{L^2}\lambda^{-2},
\end{align}
where, writing
\begin{align*}
    \de_{x_1}\rho&=\de_{x_1}(\rho + x_2),
\end{align*}
we used Sobolev embedding
\begin{align*}
    \|\de_{x_1}(\rho + x_2)\|_{L^\infty} \lesssim \|\rho+x_2\|_{H^4},
\end{align*}
and the boundedness of $\|\rho+x_2\|_{H^4}$.

The second addend in \eqref{eq:addends-new}, which involves $D^ju_2[\tilde \rho_\lambda]$, can be bounded exactly in the same way.

For the last two terms, once again we write
\begin{align*}
    \nabla D^j \rho = \nabla D^j (\rho+x_2) - D^j \begin{pmatrix}
        0\\
        1
    \end{pmatrix},
\end{align*}
so that 
\begin{align*}
    |(\mathbf{u}[\tilde \rho_\lambda] \cdot \nabla D^j \rho, D^j(\tilde \rho - \rho))_{L^2}| + |(\mathbf{u}[\tilde \rho-\rho] \cdot \nabla D^j \rho, D^j(\tilde \rho - \rho))_{L^2}|\\ \lesssim (\|\tilde \rho_\lambda\|_{L^2}+\|\tilde \rho - \rho\|_{L^2}) \|\tilde \rho - \rho\|_{H^3} (1+\|\rho+x_2\|_{H^4}).  
\end{align*}
Altogether,
\begin{align}
\frac{d}{d t}\|D^j (\tilde \rho - \rho)\|_{L^2}^2 
& \lesssim \|\rho+x_2\|_{H^3}\|\tilde \rho-\rho\|_{H^3}^2 + \lambda^{-2} (1+\|\rho+x_2\|_{H^4})\|\tilde \rho - \rho\|_{H^3} 
\end{align}
Integrating in time and applying Grönwall's inequality yields
\begin{align}\label{est:H3unscaled}
    \|\tilde \rho - \rho\|_{H^3}^2 \lesssim  \text{e}^{C(\delta, \|\rho+x_2\|_{L^\infty_t H^4}) t} \lambda^{-2} t,
\end{align}
where $C(\delta, \|\rho+x_2\|_{L^\infty_t H^4})$ denotes a constant depending on $\delta, \|\rho+x_2\|_{L^\infty_t H^4}$.
\subsubsection{Estimate of the difference $\left\|\lambda \tilde{\rho}_{\lambda,  \text{inter}}\left(\frac{x}{\lambda}, t\right)-\lambda \bar{\rho}_{\lambda,  \text{inter}}\left(\frac{x}{\lambda}, t\right)\right\|_{H^3}$.}
We drop the subscript ``inter'' for simplicity.
Consider the scaled unknowns
\begin{align}\label{eq:tilder}
    \tilde r(x, \cdot):= \lambda\tilde \rho_\lambda \left(\frac{x}{\lambda}, \cdot \right), \quad \bar r(x, \cdot):= \lambda\bar \rho_\lambda \left(\frac{x}{\lambda}, \cdot \right).
\end{align}
They satisfy the equations
\begin{align}
    \de_t \bar r + \bold u [\bar r] \cdot \nabla \bar r +  (\de_{x_1}u_1 [\rho])(x=0, t) (x_1, - x_2) \cdot \nabla \bar r&=0,   \\
    \de_t \tilde r + \bold u [\tilde r] \cdot \nabla \tilde r +  \lambda \bold u \left[\tilde \rho \left(\frac{x}{\lambda}\right)\right] \cdot \nabla \tilde r&=0.
\end{align}
Taking the difference yields
\begin{align*}
    \de_t (\bar r- \tilde r) + \bold u [\bar r] \cdot \nabla (\bar r-\tilde r) + (\bold u[\bar r]-\bold u [\tilde r]) \cdot \nabla \tilde r + (\de_{x_1}u_1 [\rho])(x=0, t) (x_1, -x_2) \cdot \nabla \bar r -  \lambda \bold u \left[\tilde \rho \left(\frac{x}{\lambda}\right)\right] \cdot \nabla \tilde r=0.
 \end{align*}
We can then use that $\tilde r=(\tilde r - \bar r) + \bar r$ and write the equation as
\begin{align}\label{eq:diffscaled}
    \de_t (\bar r- \tilde r) &+ \bold u [\bar r] \cdot \nabla (\bar r-\tilde r) + (\bold u[\bar r]-\bold u [\tilde r]) \cdot \nabla [(\tilde r-\bar r)+\bar r] - \lambda \bold u \left[\tilde \rho \left(\frac{x}{\lambda}\right)\right] \cdot \nabla (\tilde r-\bar r)\notag\\
    & -\lambda \bold u \left[\tilde \rho \left(\frac{x}{\lambda}\right)\right] \cdot \nabla \bar r + (\de_{x_1}u_1 [\rho])(x=0, t) (x_1, -x_2) \cdot \nabla \bar r =0.
 \end{align}

Taking a $D^3$ derivative, multiplying by $D^{3}(\bar r - \tilde r)$,  the first three terms are easily taken into account by a term of the form
$$\|\bar{r}-\tilde{r}\|^2_{H^3}(\|\bar{r}-\tilde{r}\|_{H^3}+\|\bar{r}\|_{H^4}+\|\tilde{\rho}\|_{H^3}).$$
The last two terms can be written as
\begin{align}
     \de_{x_1} u_1 [\rho] (x_1, -x_2) \cdot \nabla \bar r - \lambda \bold u \left[\tilde \rho \left(\frac{x}{\lambda}\right)\right] \cdot \nabla \bar r&=  \de_{x_1} u_1 [\rho-\tilde \rho] (x_1, -x_2) \cdot \nabla \bar r\notag\\
     &\quad + \left( \de_{x_1} u_1 [\tilde \rho] \left({x_1}, -{x_2}\right) -  \lambda \bold u \left[\tilde \rho \left(\frac{x}{\lambda}\right)\right]\right)\cdot \nabla \bar r\notag\\
     &=\mathrm{I}_1 + \mathrm{I}_2.
\end{align}
First, relying on \eqref{est:nonscaled-new}, the control of the support of $\rho (x, t)$ and $\tilde \rho (x, t)$ in Section \ref{sec:support} - that is at least $\delta$-far from 0 - and using Cauchy-Schwarz,
\begin{align}
    |(\mathrm{I}_1, \bar r - \tilde r)_{L^2}|& \lesssim \lambda^{-1} \|\nabla \bar r\|_{L^{\infty}} \|\bar r - \tilde r\|_{L^2},
\end{align}
and
\begin{align}
    |(\mathrm{I}_1, \bar r - \tilde r)_{H^3}|& \lesssim \lambda^{-1}\|\bar r\|_{H^4} \|\bar r - \tilde r\|_{H^3},
\end{align}
where $\|\bar r\|_{H^4}$ is uniformly bounded by hypothesis.

To deal with $\mathrm{I}_2$, we perform a Taylor expansion at first order of the velocity. That is, using the symmetry to cancel the first (0-th order) term, we write, for $x\in\text{supp}(\bar{r})$
\begin{align*}
\left|(\de_{x_1} u_1 [\tilde \rho](0,t)) \cdot \left({x_1}, -{x_2}\right)^t - \lambda \bold u \left[\tilde \rho \left(\frac{\cdot}{\lambda}\right)\right] (x,t)\right|&=\left|(\de_{x_1} u_1 [\tilde \rho](0,t)) \cdot \left({x_1}, -{x_2}\right)^t -  \lambda \bold u \left[\tilde \rho \left(\cdot\right)\right] \left(\frac{x}{\lambda},t\right)\right|\\
&\leq C\|\mathbf{u}[\tilde{\rho}]\|_{C^2(B_{\frac{c}{\lambda}}(0))}\frac{|x|^2}{\lambda}.
\end{align*}
We then have
\begin{align*}
    |(\mathrm{I}_2, \bar r - \tilde r)_{L^2}|& \lesssim 
    \|\bar r - \tilde r\|_{L^2} \sum_{j=1}^2 \lambda^{-1}\|\mathbf{u}[\tilde{\rho}]\|_{C^2}\left\| x_j^2 \de_{x_j} \bar r \right\|_{L^2} \\
    & \lesssim \lambda^{-1} \|\mathbf{u}[\tilde{\rho}]\|_{C^2(B_{\frac{c}{\lambda}}(0))}\|\bar r - \tilde r\|_{L^2}  \|\nabla \bar r\|_{L^2} \|x|^2\text{diam}(\supp(f(x))|^2 \lesssim \lambda^{-1} \|\bar r - \tilde r\|_{L^2},
\end{align*}
\begin{align*}
    |(D^{3}(\mathrm{I}_2),D^{3} (\bar r - \tilde r))_{L^2}|& \lesssim 
    \|\bar r - \tilde r\|_{H^3} \sum_{j=1}^2 \lambda^{-1}\|\mathbf{u}[\tilde{\rho}]\|_{C^3}\left\| x_j^2 \de_{x_j} \bar r \right\|_{H^3}\\
    & \lesssim \lambda^{-1} \|\mathbf{u}[\tilde{\rho}]\|_{C^3(B_{\frac{c}{\lambda}}(0))}\|\bar r - \tilde r\|_{H^3}  \|\bar r\|_{H^4} \|x|^2\text{diam}(\supp(f(x))|^2 \lesssim \lambda^{-1} \|\bar r - \tilde r\|_{H^3},
\end{align*}
where we used that $\bar r$ is uniformly bounded in $H^4$ by hypothesis, $f$ is compactly supported and, since 
$\mathbf{u}[\tilde{\rho}]=\mathbf{u}[\tilde{\rho}+x_2]$, then, by using a similar argument about separation of supports as in \eqref{eq:estu-supp-new} we get
\begin{align*}
    \|\mathbf{u}[\tilde{\rho}]\|_{C^2(B_{\frac{c}{\lambda}}(0))} \lesssim \|\mathbf{u}[\tilde{\rho}+x_2]\|_{C^2(B_{\frac{c}{\lambda}}(0))}\lesssim\|\tilde{\rho}+x_2\|_{L^2}.
\end{align*}
%
\subsubsection{Estimate of $\|\tilde \rho_{\text{ext}}(x, t)+x_2\|_{H^4}$ and $\left\|\lambda \tilde{\rho}_{\lambda, \text{inter}}\left(\frac{x}{\lambda}\right)\right\|_{H^4}$.}
Writing
\begin{align*}
\tilde \rho_{\text{ext}}-\rho = (\tilde \rho_{\text{ext}}+x_2)-(\rho+x_2), 
\end{align*}
we deduce that $\tilde \rho_{\text{ext}}+x_2$ is bounded in $H^3$ by the first inequality in \eqref{est:nonscaled-new} and by the fact that $\rho+x_2$ is bounded in $H^4$ by assumption. Now, let us look again at the equation for $\tilde \rho_{\text{ext}}$ in \eqref{eq:sum1} and notice that $\tilde \rho_{\text{ext}}+x_2$ solves
\begin{align*}
\partial_{t}(\tilde{\rho}_{\text{ext}}(x,t)+x_2)+\mathbf{u}[\tilde{\rho}_{\lambda, \text{inter}}+\tilde{\rho}_{\text{ext}}+x_2]\cdot \nabla (\tilde{\rho}_{\text{ext}}(x,t)+x_2)=u_2[\tilde{\rho}_{\text{ext}}(x,t)+x_2].
\end{align*}
Then, we have
\begin{align*}
    \frac{d}{dt}\|\tilde \rho_{\text{ext}}+x_2\|_{H^4}^2 \lesssim (\|\bold u [\tilde \rho_{\lambda, \text{inter}}]\|_{H^4(\R^2\setminus B_\delta(0))}+\|\tilde \rho_{\text{ext}}+x_2\|_{H^3}) \|\tilde \rho_{\text{ext}}+x_2\|_{H^4}^2,
\end{align*}
where $\|\bold u [\tilde \rho_{\lambda, \text{inter}}]\|_{H^4(\R^2\setminus B_\delta(0))}$ is uniformly bounded by \eqref{eq:estu-supp-new} and $\|\tilde \rho_{\text{ext}}+x_2\|_{H^3}$ is also uniformly bounded, as just observed. The bound follows by Grönwall inequality.
The same approach as before applies also for $\left\|\lambda \tilde{\rho}_{\lambda, \text{inter}}\left(\frac{x}{\lambda}\right)\right\|_{H^4}$, 
writing the equation for $\tilde{\rho}_{\lambda, \text{inter}}$ in the form
\begin{align*}
\partial_{t}\tilde{\rho}_{\lambda, \text{inter}}(x,t)+\bold{u}[\tilde{\rho}_{\lambda, \text{inter}}+\tilde{\rho}_{\text{ext}}+x_2]\cdot \nabla \tilde{\rho}_{\lambda, \text{inter}}=0,
\end{align*}
and using that $\| \tilde{ \rho}_{\lambda, \text{inter}}\|_{H^3}$ and $\|\tilde \rho_{\text{ext}}+x_2\|_{H^4}$ are uniformly bounded. Continuity in time  $C([0,T]; H^4 (\R^2))$ then follows by standard results from well posedness in $H^{4}$ for IPM.
This concludes the proof.
%
%
%
\end{proof}


The simplified equation \eqref{simp1} in Lemma \ref{gluing1} remains challenging to solve as, a priori, we grapple with the difficulty introduced by the quadratic term in \eqref{simp1}. In the next lemma, we will obtain some properties for an approximation of \eqref{eq:ipm}.
\begin{lem}\label{gluing2}
   Given a constant $M>0$, a time $T>0$ and $k(t)<0$ for all $t\in[0,T]$ such that 
   $M\geq \int_{0}^{T}|k(t)| \, d t$,
   and $\rho_{\text{in}}(x)\in H^{4}$ compactly supported, then there exist  $C,a_{0}>0$ such that, for $a\leq a_{0}$, if we define
\begin{align}
    \de_t \tilde{\rho} + (k(t)(x_{1},-x_{2})) \cdot \nabla \tilde{\rho}&=0,\notag \\
    \de_t \bar{\rho} + (k(t)(x_{1},-x_{2})+\bold u [\bar{\rho}]) \cdot \nabla \bar{\rho}&=0,\notag \\
    \bar{\rho}(x,t=0)=\tilde{\rho}(x,t=0)&=a\rho_{\text{in}}(x),\notag \\
\end{align}
then, for $t\in[0,T]$, we have $\|\tilde{\rho}-\bar{\rho}\|_{H^3}\leq C a^2$, and furthermore $\de_{t}\|\tilde{\rho}-\bar{\rho}\|_{H^3}\leq C a^2.$
\end{lem}
\begin{proof}
We start by noting that, taking derivatives $D^{3}$ and multiplying by $D^{3}\bar{\rho}$ in the evolution equation for $\rho$ we get
$$\frac{1}{2}\frac{d}{d t}\|\bar \rho\|_{H^4}^2 \lesssim  |k(t)|\|\bar \rho\|^{2}_{H^4}+\|\bar \rho\|^{3}_{H^4} $$
so using $\|\bar{\rho}(x,0)\|_{H^4}\leq Ca$ a Gronwall type estimate gives, if $a$ is small enough that
$$\|\bar \rho\|_{H^4}\leq e^{2M}\|\bar{\rho}(x,0)\|_{H^4}=Ca.$$
    If we now consider the equation for the difference
    \begin{align}
        \de_t (\tilde \rho-\bar \rho) +  (k(t)(x_{1},-x_{2})) \cdot \nabla (\tilde \rho-\bar \rho) - \bold u (\bar \rho) \cdot \nabla \bar \rho =0,
    \end{align}
    and again we take derivatives $D^{3}$ and the scalar product against $D^{3}(\tilde \rho - \bar \rho)$,  this yields
    \begin{align}
        \frac{1}{2}\frac{d}{d t}\|\tilde \rho-\bar \rho\|_{H^3}^2 &\lesssim |k(t)|\|\tilde \rho-\bar \rho\|^{2}_{H^3}+\|\bold u (\bar \rho) \cdot \nabla \bar \rho \|_{H^3} \|\tilde \rho-\bar \rho\|_{H^3} \notag \\
        &\lesssim  |k(t)|\|\tilde \rho-\bar \rho\|^{2}_{H^3}+\|\bar \rho \|^2_{H^4} \|\tilde \rho-\bar \rho\|_{H^3}.
    \end{align}
    %
    %
    %
    %
    %
    %
    which, again, by a Gronwall type estimate, yields the desired estimate
    \begin{align}
        \|\tilde \rho-\bar \rho\|_{H^4}\lesssim C a^2,  
    \end{align}
    and going back to the evolution equation for the $H^3$ norm gives the desired estimate for the derivative of the norm.
    %
    %
    %
\end{proof}
%

\begin{lem}\label{inductiondeform}
    Given $\epsilon_{1},\epsilon_{2},\epsilon_{3},T,M,K,d>0$, with $K>1$,  $\epsilon_{2},\epsilon_{3} <1$, and a symmetric solution $\rho(x,t)\in C^\infty$ to \eqref{eq:ipm} for $t\in[0,T]$, with stable initial conditions \eqref{eq:data-stable}, such that $\rho(x,t)+x_2 \in H^4$ and $\supp(\rho(x,t))\cap B_{\delta}(0)=\emptyset$ for some $\delta>0$, and
    \begin{align*}
        k(t)&=\de_{x_{1}}u_1[\rho](x=0,t)<0,\quad -\int_{0}^{T}k(t) \, d t \leq M,\\
        \de_{t}k(t)&>-d, \quad \|\rho(x,t)+x_2\|_{C^1}<K, \quad \|\rho(x,t=0)+x_2\|_{H^2}< \epsilon_{1}.
    \end{align*}
    Then, we can find a function $\rho_{pert}(x, 0)$ with 
    $$\supp(\rho(x,0))\cap\supp (\rho_{pert}(x, 0))=\emptyset,$$ 
    such that the solution $\rho_{new}(x,t)$ to \eqref{eq:ipm} with initial conditions $\rho(x,t=0)+\rho_{pert}(x, 0)$ is symmetric,  $ \rho_{new}(x, t)+x_2 \in H^4$,  $\supp(\rho_{new}(x,t))\cap B_{\delta_{new}}(0)=\emptyset$ for some $\delta_{new}>0$, and
    \begin{align}
        k_{new}(t):&=\de_{x_{1}}u_1[\rho_{new}](x=0,t)<0,\quad  \int_{0}^{T}k_{new}(t)\, d t \leq \int_{0}^{T}k(t)\, d t-\epsilon_{2}c_{M},\quad 
        \de_{t}k_{new}(t)>-d, \label{eq:estknew}\\
        \|\rho_{new}(x,t)+x_2\|_{C^1}&<K,\quad \|\rho_{pert}(x)\|_{H^2}< \epsilon_{2}, \quad\|\rho_{new}(x,t)-\rho(x,t)\|_{C^1}<\epsilon_{2}, \label{est:rhonew}
    \end{align}
with $c_{M} > 0$, a constant that depends only on $M$ and decreases as $M$ increases.
Furthermore, we have that
\begin{align}
    \|\rho(x,t)-\rho_{new}(x,t)\|_{C^{2.5}(\R^2\setminus B_{\delta}(0))}&\leq \epsilon_{3},\\
    \|\rho(x,t)-\rho_{new}(x,t)\|_{L^{2}(\R^2)}&\leq \epsilon_{3}.
\end{align}
\end{lem}
\begin{proof}

%
We consider an initial condition $\rho(x,t=0)+\rho_{pert}(x,t=0)$, where $\rho_{pert}(x,t=0)=a\frac{f(\lambda x)}{\lambda}$, with $f(x)$ symmetric, compactly supported, fulfilling the support conditions of Lemma \ref{stabledeformation}, together with $f(x) \cap B_{\frac 14}(0)=\emptyset$, and $a \le a_0$ as in Lemma \ref{gluing2}.
First, note that for $\lambda$ so big that 
\begin{align*}
    \frac{\text{diam}(\supp (f(x))}{\lambda} < \delta,
\end{align*}
it holds
\begin{equation}\label{eq:suppinitiallemmadeform}
    \text{supp}(\rho(x,0))\cap\text{supp} (\rho_{pert}(x, 0))=\emptyset.
\end{equation}
Now, we would like to prove that 
$\rho(x,t)+{\rho}_{pert}(x,t)$, with ${\rho}_{pert}(x,t)$ solving \eqref{eq:lemmastabledef}, is a good approximation of $\rho_{new}(x,t)$.
It is easy to see that for $t \in [0,T]$, adjusting $\lambda$ if needed,
\begin{equation}
    \text{supp}(\rho(x,t))\cap\text{supp} (\rho_{pert}(x, t))=\emptyset.
\end{equation}
By \eqref{ineq:lemma1.3-lower} in Lemma \ref{stabledeformation}, we have %
\begin{align}
    \int_{0}^{T} \de_{x_{1}}u_1[\rho+{\rho}_{pert}](x,t)\, d t & \leq \int_{0}^{T} \de_{x_{1}}u_1[\rho](x,t) \, d t+{\text{e}^{-7M}T}\de_{x_1}u_1[\tilde \rho_{pert}](x=0, t=0)\notag\\
    & \le \int_{0}^{T} \de_{x_{1}}u_1[\rho](x,t) \, d t-\tilde Ca {\text{e}^{-7M}T},\label{est:k}
\end{align}
where we used the scaling $\tilde \rho_{pert}(x, t=0)=\frac{f(\lambda x)}{\lambda}$ (where $f(\lambda x)$ is supported away from a ball centered in the origin of radius $\frac{1}{4 \lambda}$), so that, in polar coordinates $(r', \alpha')$
\begin{align*}
    0>\de_{x_1}u_1[\tilde \rho_{pert}](x=0, t=0) =: -\tilde C
\end{align*}
is uniformly bounded in $\lambda$.
As ${\rho}_{pert}(x,t)$ solves \eqref{eq:lemmastabledef}, relying on Lemma \ref{gluing2} we can introduce the approximation $\bar \rho_{pert}(x,t)$, satisfying 
\begin{align}
    \de_t \bar \rho_{pert} + (k(t)(x_1, -x_2) + \mathbf{u}[\bar \rho_{pert}])\cdot \nabla \bar \rho_{pert}=0,
\end{align}
with the same initial condition $\bar \rho_{pert}(x, t=0)=\rho_{pert}(x, t=0)$. In fact, from Lemma \ref{gluing2} and Sobolev embedding, it follows that, after rescaling in $\lambda$,
\begin{align}\label{est:pert}
    \| (\bar \rho_{pert}- {\rho}_{pert})\left(t, {\cdot}\right)\|_{C^1} &= 
    \left\| \lambda\left(\bar \rho_{pert}- {\rho}_{pert}\right)\left(t, \frac{\cdot}{\lambda}\right)\right\|_{C^1} \lesssim \left\|\lambda(\bar \rho_{pert}- {\rho}_{pert})\left (t, \frac{\cdot}{\lambda}\right)\right\|_{H^3} \lesssim a^2 t, \\\ \de_t \|(\bar \rho_{pert}- {\rho}_{pert})(t)\|_{C^1} &\lesssim a^2,
\end{align}
where $\lesssim$ hides a constant independent of $\lambda, a$.

Now we appeal to Lemma \ref{gluing1}. Specifically, as $\bar \rho_{pert}(x, t)$ solves \eqref{simp1}, we now introduce two functions $\tilde \rho_{pert}(x, t)$ - solving the last system in Lemma \ref{gluing1} - and $\tilde \rho (x,t)$  - solving the second-to-last system in Lemma \ref{gluing1}. This way, as pointed out after Lemma \ref{gluing1}, $\tilde \rho(x, t) + \tilde \rho_{pert} (x,t)$ is an exact solution to \eqref{eq:ipm} with initial data 
\begin{align}
    \rho(x, t=0) + \tilde \rho_{pert} (x,t=0)=\rho_{new}(x, t=0),
\end{align}
where $\rho_{new}(x,t)\in C^\infty$ is an exact solution to \eqref{eq:ipm} as well, with the same initial conditions. By uniqueness, $\rho_{new}(x,t)$ and $\tilde\rho(x, t) + \tilde \rho_{pert} (x,t)$ must coincide at least for $t \in [0,T]$. Notice that, if the hypothesis for Lemma \ref{gluing1} are fulfilled, it holds that
\begin{align}\label{eq:bounds-small}
    \|\tilde \rho - \rho\|_{H^3} \le C\lambda^{-1}, \quad \|\tilde \rho + x_2\|_{H^4} \le C, \quad \|\rho + x_2\|_{H^4} \le C.
\end{align}
Therefore, it remains to prove that $\bar \rho_{pert} (x,t)$ as above is a good approximation to $\tilde \rho_{pert} (x,t)$.
To apply Lemma \ref{gluing1}, we need to verify that $\left\|\lambda \bar\rho_{pert}\left(\frac{x}{\lambda}, t\right) \right\|_{H^4}$ is uniformly bounded for all $t \in [0,T]$, where $T>0$ is arbitrarily given by the statement of Lemma \ref{inductiondeform}. To this end, first note that, by standard commutator estimates,
\begin{align*}
    \frac{d}{dt}\left\|\lambda \bar \rho_{pert}\left(\frac{x}{\lambda}, t\right)\right\|_{H^4}^2 \lesssim \left(-k(t)+\left\|\lambda \bar \rho_{pert}\left(\frac{x}{\lambda}, t\right)\right\|_{H^4}\right) \left\|\lambda \bar \rho_{pert}\left(\frac{x}{\lambda}, t\right)\right\|_{H^4}^2,
\end{align*}
which gives, simplifying the squares and introducing $\bold y(t):=\text{e}^{\int_0^t k (\tau) \, d\tau}\left\|\lambda \bar \rho_{pert}\left(\frac{x}{\lambda}, t\right)\right\|_{H^4}$, the ODE
\begin{align*}
    \frac{d}{dt}{\bold y(t)} \lesssim \text{e}^{M}{\bold y(t)}^2
\end{align*}
with initial data
\begin{align}
    \bold y(0)=\left\|\lambda \bar \rho_{pert}\left(\frac{x}{\lambda}, t=0\right)\right\|_{H^4} \lesssim a.
\end{align}
Then, the existence time of $\left\|\lambda \bar \rho_{pert}\left(\frac{x}{\lambda}, t\right)\right\|_{H^4}$ is lower bounded as
\begin{align}
    T_{\bar \rho_{pert}} \ge C^{-1} a^{-1},
\end{align}
where $C$ is independent of $a$,
so that we can choose $a$ small enough to deduce that $\left\|\lambda \bar \rho_{pert}\left(\frac{x}{\lambda}, t\right)\right\|_{H^4} $ is uniformly bounded for all $t \in [0,T]$. 
We can thus apply Lemma \ref{gluing1}.
From \eqref{est:scaled-new}, it follows that
\begin{align}
    \left\|\lambda \tilde{\rho}_{pert}\left(\frac{x}{\lambda}\right)-\lambda \bar{\rho}_{pert}\left(\frac{x}{\lambda}\right)\right\|_{H^3}\leq C t\lambda^{-1},\quad \de_{t}\left\|\lambda \tilde{\rho}_{pert}\left(\frac{x}{\lambda}\right)-\lambda \bar{\rho}_{pert}\left(\frac{x}{\lambda}\right)\right\|_{H^3}\leq C\lambda^{-1}.
\end{align}
Note that due to the $C^1$ norm behaviour under the scaling $\lambda f (\cdot/\lambda)$, we have
\begin{align*}
     \left\| \tilde{\rho}_{pert}\left({x}\right)- \bar{\rho}_{pert}\left({x}\right)\right\|_{C^1} & \leq  \left\|\lambda \tilde{\rho}_{pert}\left(\frac{x}{\lambda}\right)-\lambda \bar{\rho}_{pert}\left(\frac{x}{\lambda}\right)\right\|_{C^1} \\
     & \le C \left\|\lambda \tilde{\rho}_{pert}\left(\frac{x}{\lambda}\right)-\lambda \bar{\rho}_{pert}\left(\frac{x}{\lambda}\right)\right\|_{H^3} \le C \lambda^{-1} t,
\end{align*}
by Lemma \ref{gluing1}. The estimate of the time derivative works similarly.
Altogether, and using additionally that 
$$\|\de_{x_{i}}u_j[g(x)]\|_{L^{\infty}}=\left\|\de_{x_{i}}\left(u_j\left[\frac{g(\lambda x)}{\lambda}\right]\right)\right\|_{L^{\infty}}, \quad i, j \in \{1,2\},$$
we obtain the following estimates:
\begin{align}
    \de_{t}\|(\rho_{new}-\tilde \rho-\tilde{\rho}_{pert})(t)\|_{C^1}&\leq C(\lambda^{-1}+a^2),\notag\\
    \|(\rho_{new}-\tilde\rho-\tilde{\rho}_{pert})(t)\|_{C^1}&\leq Ct(\lambda^{-1}+a^2), \label{re:C1close}\\
    \de_{t}\|\mathbf{u}[\rho_{new}-\tilde\rho-\tilde{\rho}_{pert}](t)\|_{C^1}&\leq C(\lambda^{-1}+a^2),\notag\\
    \|\mathbf{u}[\rho_{new}-\tilde\rho-\tilde{\rho}_{pert}](t)\|_{C^1}&\leq Ct(\lambda^{-1}+a^2).\notag
\end{align}
Furthermore, from Lemma \ref{stabledeformation}, we have that
$$\|\tilde{\rho}_{pert}(x,t)\|_{C^1}\leq Ca\mathrm{e}^{M},\quad \de_{t}\de_{x_{1}}u_1[\tilde{\rho}_{pert}](x=0,t)\geq 0.$$ 
All of these conditions together and \eqref{est:k} imply that there exists some small $a_{0}(M)$ depending only on $M$ such that,for $a\leq a_{0}$, for $\lambda$ big enough,  we have for $t\in[0,T]$ that $\rho_{new}(x,t)\in H^3$ with
\begin{align*}
k_{new}(t)&=\de_{x_{1}}u_1[\rho_{new}](x=0, t)<0,\quad \int_{0}^{T}k_{new}(t)\, d t\leq \int_{0}^{T}k(t)\, d t-{\tilde C}a\mathrm{e}^{-7M},\\
\de_{t}k_{new}(t)&>-d, \quad \|\rho_{new}(x,t)+x_2\|_{C^1}<K, \quad \|\rho_{new}(x,t)-\rho(x,t)\|_{C^1}<\bar{C}(\lambda^{-1}+a^2+ae^{M}).
\end{align*}
In fact, assuming without loss of generality $a_{0}(M)\leq 1$, and after taking $\lambda$ big enough,  we simply set 
$$a= a_{0}(M)\frac{e^{-M}}{\bar{C}}\frac{\eps_{2}}{3},$$
so that we obtain the desired inequalities for $\|\rho_{new}\|_{C^1}$ and $\int_0^T k(t) \, d t$ upon choosing 
$c_{M}=a_{0}(M)\frac{{\tilde C}}{3\bar{C}}\mathrm{e}^{-8M}$. 
Next, regarding the support, note that the property holds true at $t=0$. Since $\mathbf{u}[\rho_{new}](x=0)=0$, the control of the ${C}^1$ norm of the velocity $\|\mathbf{u}[\rho_{new}+x_2]\|_{C^1} \le C$ (which follows by the $H^3$ control of $\tilde \rho (x, t)+x_2$ and $\tilde \rho_{pert}(x, t)$ by \eqref{est:nonscaled-new}-\eqref{est:scaled-new}) implies that our approach to the origin is only exponential. 
In particular, this implies that $$\supp(\rho_{new}(x,t))\cap B_{\delta_{new}}(0)=\emptyset,$$ 
with
\begin{align*}
    \delta_{new} \ge \frac{c}{\lambda} \text{e}^{Ct}, \quad t \in [0,T], 
\end{align*}
and ensures that $\supp (\tilde \rho) \cap \supp (\tilde \rho_{pert}) = \emptyset$, by
choosing $\lambda$ so big that $\delta > \frac{c}{\lambda} \text{e}^{CT}$. 
Now consider, for $t \in [0, T]$,
\begin{align*}
    \|(\rho_{new}-\rho)(t)\|_{C^{2.5}(\R^2\setminus B_{\delta}(0))} & =  \|(\tilde \rho+\tilde \rho_{pert}-\rho)(t)\|_{C^{2.5}(\R^2\setminus B_{\delta}(0))} \\
    &\le \|(\tilde \rho-\rho)(t)\|_{C^{2.5}(\R^2\setminus B_{\delta}(0))},
\end{align*}
where the second addend cancels since $\supp \tilde \rho_{pert} \subset B_\delta (0)$.
We use that
\begin{align*}
    \|(\tilde \rho - \rho)(t)\|_{C^{2.5}(\R^2\setminus B_\delta(0))} \lesssim \|(\tilde \rho - \rho)(t)\|_{H^{3.6}(\R^2)},
\end{align*}
where
\begin{align*}
    \|(\tilde \rho - \rho)(t)\|_{H^{3.6}} \le C\lambda^{-\frac{2}{5}}
\end{align*}
by interpolation between the control of $\|(\tilde \rho - \rho)(t)\|_{H^{3}} \lesssim C \lambda^{-1}$ and $\|\tilde \rho-\rho\|_{H^4}=\|(\tilde \rho+x_2)-(\rho+x_2)\|_{H^4} \le C$ in \eqref{eq:bounds-small}.

Finally note that
$$\|\rho_{new}-\rho\|_{L^2}\leq \|\tilde{\rho}-\rho\|_{L^2}+\|\tilde{\rho}_{pert}\|_{L^2}\leq C\lambda^{-1}.$$
The proof is concluded.
\end{proof}


The previous lemma allows us to construct solutions with a small norm in $H^2$ whose velocity generates a strong deformation near the origin. To establish our final result, we require another lemma that leverages this deformation to induce rapid growth in $H^2$. However, before proving this, we need to introduce a brief auxiliary result.

\begin{lem}\label{boundvelocity}
    Given $R>0$ and $j\in \N\cup 0$, there exists a constant $C>0$ such that, if $f(x)$ is a $C^{\infty}$ function with support in  $B_{R}(0)$, then, for any $A>2$ and any $\theta_0 \in \R$,
    $$\|u_{1}[f(x)\sin(A x_{1}+\theta_{0})]\|_{C^{j}} \leq C\ln(A)A ^{j-1}\|f\|_{C^{j+2}}$$
    $$\|u_{2}[f(x)\sin(A x_{1}+\theta_{0})]\|_{C^{j}} \leq C\ln(A)A ^{j}\|f\|_{C^{j+1}}$$
    $$\|u_{1}[f(x)\sin(A x_{2}+\theta_{0})]\|_{C^{j}} \leq C\ln(A)A ^{j}\|f\|_{C^{j+1}}$$
    $$\|u_{2}[f(x)\sin(A x_{2}+\theta_{0})]\|_{C^{j}} \leq C\ln(A)A ^{j-1}\|f\|_{C^{j+2}}.$$
\end{lem}
\begin{proof}
    First, let us recall that $\partial_{x_{i}} \mathbf{u}[\rho]=\mathbf{u}[\partial_{x_{i}}\rho]$ and thus if we prove the bounds in the $L^{\infty}$ case, we can obtain the estimates in $C^{j}$ by taking derivatives. We will focus on the second component of the velocity, the other case being analogous. To show the required bound, we use that
    $$\mathbf{u}[\rho](x)=-\frac{1}{2\pi}\int_{\R^2}\frac{(x_{1}-y_{1})}{|x-y|^2}(\nabla^{\perp}\rho)(y)\, dy.$$
    Notice that
    $$ \left|\int_{|h|\leq \frac{1}{A}}\frac{h_{1}}{|h|^2}g(x+h)\sin(Ax_{i}+Ah_{i}+\theta_{0})\, dh\right|\leq  \frac{C\|g\|_{L^{\infty}}}{A}, $$
    and integrating by parts with respect to $h_{1}$,
    \begin{align*}
        \left|\int_{1\geq |h|\geq \frac{1}{A}}\frac{h_{1}}{|h|^2}g(x+h)\sin(Ax_{1}+Ah_{1}+\theta_{0})\, dh\right|
        &\leq C\int_{1\geq |h|\geq \frac{1}{A}}\left(\frac{1}{|h|^2}\|g\|_{L^{\infty}}+\frac{1}{|h|}\|g\|_{C^1}\right)\frac{1}{A}\, dh\\
        &\quad +C\int_{1\geq |h_{2}|
        \geq \frac{1}{A}}\frac{1}{|h_{2}|}\frac{\|g\|_{L^{\infty}}}{A}\, dh_{2}\\&\leq C\ln(A)\frac{\|g\|_{C^1}}{A}.
    \end{align*}
Next, using again integration by parts
    \begin{align*}
        \left|\int_{|h|\geq 1}\frac{h_{1}}{|h|^2}g(x+h)\sin(Ax_{1}+Ah_{1}\theta_{0})\,dh\right|&=\left|\int_{B_{R}(-x)\cap |h|\geq 1}\frac{h_{1}}{|h|^2}g(x+h)\sin(Ax_{1}+Ah_{1}\theta_{0})\, dh\right|\\
        &\leq C\int_{B_{R}(-x)\cap |h|\geq 1}\left(\frac{\|g\|_{L^{\infty}}}{|h|^2}+\frac{\|g\|_{C^1}}{|h|}\right)\frac{1}{A}\, dh\\
        &\quad +C\int_{0}^{2R}\frac{\|g\|_{L^{\infty}}}{A}\, dh_{2}\\
        &\leq C\frac{\|g\|_{C^1}}{A}\int_{B_{R}(-x)\cap |h|\geq 1}dh\\
        &\quad +C\frac{\|g\|_{L^{\infty}}}{A}\\
        &\leq \frac{C\|g\|_{C^1}}{A},
    \end{align*}
        so that
        $$\frac{1}{2\pi}\left\|\int_{\R^2}-\frac{(x_{1}-y_{1})}{|x-y|^2}(\nabla^{\perp}\rho)(y)\, dy\right\|_{L^{\infty}}\leq \frac{C\ln(A)\|g\|_{C^1}}{A}.$$
        This way,
        \begin{align*}
            \|u_{2}[f(x)\sin(A x_{1}+\theta_{0})]\|_{L^{\infty}}
            &\leq C\left|\int_{\R^2}\frac{(x_{1}-y_{1})}{|x-y|^2}((\partial x_{1}f)(y)\sin(A y_{1}+\theta_{0}))\,dy\right|\\
            &\quad +C\left|\int_{\R^2}\frac{(x_{1}-y_{1})}{|x-y|^2}(f(y)A\cos(A y_{1}+\theta_{0}))\, dy\right|\leq C\ln(A)\|f\|_{C^2},
        \end{align*}
         concluding the proof.

\end{proof}
\section{Iterative construction}
We now start the iterative construction of our approximate solution with a strong deformation at the origin.
\begin{lem}\label{inductiongrowth}
    Given $\epsilon_{1},\epsilon_{2},\epsilon_{3},K,T,M,d>0$, and  $\rho(x,t)$ a symmetric solution to \eqref{eq:ipm} for $t\in[0,T]$, such that at the initial time $\rho(x, t=0)=-x_2+\rho_{\text{in}}(x)$ as in \eqref{eq:data-stable} (stable initial data), $\rho(x,t)+x_2\in H^4$,  $\supp(\rho(x,t))\cap B_{\delta}(0)=\emptyset$ for some $\delta>0$, and such that
    $$k(t):=\de_{x_{1}}u_1[\rho](x,t)<0,\quad  -\int_{0}^{T}k(t)\, d t= M,$$
    $$\de_{t}k(t)>-d, \quad \|\rho(x,t)+x_2\|_{C^1}<K,\quad \|\rho(x,t=0)+x_2\|_{H^2}< \epsilon_{1}.$$
    Then, we can find a function $\rho_{pert}(x,t)$ such that, the solution $\rho_{new}(x,t)$ to \eqref{eq:ipm} with initial conditions $\rho(x,t=0)+\rho_{pert}(x,t=0)$ is symmetric, $\rho_{new}(x,t)+x_{2}\in H^4$ for $t\in[0,T]$, $\supp(\rho_{new}(x,t))\cap B_{\delta_{new}}(0)=\emptyset$ for some $\delta_{new}>0$, and
    \begin{equation}\label{growth1}
        k_{new}(t):=\de_{x_{1}}u_1[\rho_{new}(x,t)](x=0)<0,\ \int_{0}^{T}k_{new}(t)\, d t\leq \int_{0}^{T}k(t)\, d t+\epsilon_{3},
    \end{equation}
    \begin{equation}\label{growth2}
        \de_{t}k_{new}(t)>-d, \; \|\rho_{new}(x,t)+x_2\|_{C^1}<K,\;\|\rho_{pert}(x,t=0)\|_{H^2}< \epsilon_{2}, \; \|\rho_{new}(x,t)-\rho(x,t)\|_{C^1}<\epsilon_{3},
    \end{equation}
\begin{equation}\label{growth3}
    \|\rho_{new}(x,t)+x_2\|_{H^2}\geq \frac{\epsilon_{2}}{10}\text{e}^{\frac{2Mt}{T}-Ttd}.
\end{equation}
Finally, we have that
\begin{align}\label{convh2K}
    \|\rho(x,t)-\rho_{new}(x,t)\|_{C^{2.5}(\R^2\setminus B_{\delta}(0))}&\leq \epsilon_{3}; \\
    \label{convL2}
    \|\rho(x,t)-\rho_{new}(x,t)\|_{L^{2}(\R^2)}&\leq \epsilon_{3}.
\end{align}

\end{lem}
\begin{Rmk}
    We choose to prove the convergence in $C^{2.5}$ for \eqref{convh2K}, but this could be done in any $C^{2,\alpha}$ with $\alpha\in(0,1)$.
\end{Rmk}

\begin{proof}
    We start by considering 
    $$\rho_{pert,N}(x)=f(Nx)\frac{\sin(N^{1+\frac{1}{10}}x_1)}{N^{1+\frac{1}{5}}},$$
    with $f(x)$ a fixed, smooth, compactly supported function such that $f(x)=f(-x)$, and $\text{supp}f\subset B_{1}(0)\setminus B_{\frac{1}{10}}(0)$. 
    We want to show that, for $N$ big enough, the solution $\rho_{new,N}(x,t)$ with initial conditions $\rho(x,0)+\rho_{pert,N}(x)$ has the desired properties.
    Note that $\rho_{new,N}(x,t)+x_2\in H^{3}$ at least for some short period of time, and for those times we can write $\rho_{new,N}(x,t)=\rho_{N}(x,t)+\rho_{pert,N}(x,t)$
    with
    \begin{align}
        \partial_{t}\rho_{N}(x,t)+\mathbf{u}[\rho_{N}+\rho_{pert,N}](x,t)\cdot\nabla\rho_{N}(x,t)&=0, \notag\\
        \rho_{N}(x,0)&=-x_2+\rho(x,0)
    \end{align}
    and 
    \begin{align}
        \partial_{t}\rho_{pert,N}(x,t)+\mathbf{u}[\rho_{N}+\rho_{pert,N}](x,t)\cdot\nabla\rho_{pert,N}(x,t)&=0,\notag\\
        \rho_{pert,N}(x,0)&=f(Nx)\frac{\sin(N^{1+\frac{1}{10}}x_1)}{N^{1+\frac{1}{5}}}.
    \end{align}
    We then hope that $\rho(x,t)$ is a good approximation of $\rho_{N}(x,t)$, and if we define
    \begin{align}
    \partial_{t}\bar{\rho}_{pert,N}(x,t)+(\partial_{x_{1}}u_{1}[\rho_{N}](x=0, t))(x_{1},-x_{2})\cdot\nabla\bar{\rho}_{pert,N}(x,t)&=0, \\
    \bar{\rho}_{pert,N}(x,0)&=f(Nx)\frac{\sin(N^{1+\frac{1}{10}}x_1)}{N^{1+\frac{1}{5}}},
    \end{align}
    that $\bar{\rho}_{pert,N}(x,t)$ is a good approximation of $\rho_{pert,N}(x,t)$.

    We will start by showing some control for the times $t\in[0,t_{N}]$, where $t_{N}$ is the maximum time such that, for $t\in[0,t_{N}]$, it holds that
    \begin{itemize}
        \item $\text{sup}_{t\in[0,t_{N}]}\|\rho_{N}(x,t)-\rho(x,t)\|_{H^3}\leq 1$ for $t\in[0,t_{N})$,
        \item $\|N\rho_{pert,N}(\frac{x}{N},t)-N\bar{\rho}_{pert,N}(\frac{x}{N},t)\|_{ H^{3}}\leq 1$ for $t\in[0,t_{N})$,
        \item $\|N\bar{\rho}_{pert,N}(\frac{x}{N},t)\|_{H^3}\leq \text{e}^{3(M+T)}\|N\bar{\rho}_{pert,N}(\frac{x}{N},t=0)\|_{H^3}$ for $t\in[0,t_{N})$,
    \end{itemize}
    and note that $t_{N}>0$ and that for the times considered we have  $\rho_{new,N}(x,t)+x_2\in H^3$. We will assume, without loss of generality, that $t_{N}\leq T$.
    Note that, by the evolution equation for $\rho_{N,pert}$, the bounds for $k(t)$ and the (assumed) control of $\|\rho_{N}(x,t)-\rho(x,t)\|_{H^3}$ we can bound
    $$\|N\bar{\rho}_{pert,N}(\frac{x}{N},t)\|_{H^3}\leq \text{e}^{3\int_{0}^{T}\partial_{x_{1}}u_{1}[\rho_{N}](x=0, t)dt}\|N\bar{\rho}_{pert,N}(\frac{x}{N},t=0)\|_{H^3}\leq \text{e}^{3(M+T)}\|N\bar{\rho}_{pert,N}(\frac{x}{N},t=0)\|_{H^3}.$$
    
    Using the properties for $t\in[0,t_{N}]$, if we define $\bar{K}:=\text{sup}_{t\in [0,t_{N}]}\|\mathbf{u}[\rho_{new}](x,t)\|_{C^1}$, then, using that $\mathbf{u}[x_2]=0$, that the scaling $N f(\cdot/N)$ can only make the $C^1$ norm grow for $\lambda\geq 1$ and $H^3 \hookrightarrow C^1$, 
    \begin{align*}
        \bar{K}&\leq \|\mathbf \mathbf{u}[\rho_{N}](x,t)\|_{C^1}+\|\mathbf \mathbf{u}[\rho_{pert,N}](x,t)\|_{C^1}=\|\mathbf \mathbf{u}[\rho_{N}+x_2 ](x,t)\|_{C^1}+\|\mathbf \mathbf{u}[\rho_{pert,N}](x,t)\|_{C^1}\\ &
        \|\mathbf \mathbf{u}[\rho_{N}+x_2 ](x,t)\|_{C^1}+\|N\mathbf \mathbf{u}[\rho_{pert,N}](\frac{x}{N},t)\|_{C^1}\leq\|\rho_{N}(x,t)+x_2\|_{H^3}+ \|N\rho_{pert,N}\left(\frac{x}{N},t\right)\|_{ H^{3}}\\
        &\leq \|\rho_{N}(x,t)+x_2\|_{H^3}+ \left\|N\bar{\rho}_{pert,N}\left(\frac{x}{N},t\right)\right\|_{ H^{3}}+\left\|N\rho_{pert,N}\left(\frac{x}{N},t\right)-N\bar{\rho}_{pert,N}\left(\frac{x}{N},t\right)\right\|_{ H^{3}}\leq C,
    \end{align*}
    where $C>0$ does not depend on $N$, so in particular, since $\mathbf u [\rho](x=0)$, the support of $\rho_{pert, N}(x,t)$ is driven by the equation
    \begin{align*}
        \dot \Phi(x, t) = \mathbf{u}[\rho_N+\rho_{pert, N}](\Phi (x, t))
    \end{align*}
    and by the Cauchy-Lipschitz Theorem 
    \begin{align*}
        |\Phi(x, 0)| \text{e}^{-Ct} \lesssim |\Phi(x, t)| \lesssim  |\Phi(x, 0)| \text{e}^{Ct}.
    \end{align*}
    Therefore, 
    \begin{equation}\label{supportpertN}
        \text{supp}(\rho_{pert,N})\subset\{x:|x|\leq { N^{-1}}\text{e}^{Ct}\}
    \end{equation}
    and $\text{supp}(\rho_{N})\cap B_{\delta}(0)=\emptyset$ for some small $\delta>0$,
    so that for $N$ big we have separation between the supports of $\rho_{pert,N}$ and $\rho_{N}$ and therefore, for any $i=0,1,2,3$, for $|x|\in \text{supp}(\rho_{N})$, and calling $D^{i}$ an arbitrary derivative of order $i$
    $$|D^{i}\mathbf{u}[\rho_{pert,N}](x)|\leq C\frac{\|\rho_{pert,N}\|_{L^1}}{|x|^{2+i}}\leq \frac{C_{i}}{|x|^{2+i}N^3},$$
    and note that the same bound holds for $|x|\in \text{supp}(\rho)$.
    Now, defining $W:=\rho_{N}(x,t)-\rho(x,t)$, we have that
    $$\partial_{t}W+\mathbf{u}[W+\rho]\cdot\nabla W+\mathbf{u}[W]\cdot\nabla \rho+\mathbf{u}[\rho_{pert,N}]\cdot\nabla(W+\rho)=0.$$
    This system allows us to obtain bounds for the evolution of the $H^{3}$ norm of $W$, namely
    $$\partial_{t}\|W\|_{H^3}\leq C[\|W\|_{H^3}(\|W\|_{H^3}+\|\rho(x,t)+x_2\|_{H^4}+\|\mathbf{u}[\rho_{pert,N}]\|_{C^3(\text{supp}(W))})+\|\mathbf{u}[\rho_{pert,N}]\cdot\nabla\rho\|_{H^3}]. $$
    But then, using that
    $$\|\mathbf{u}[\rho_{pert,N}]\cdot\nabla\rho\|_{H^3}\leq \frac{C}{N^3}(1+\|\rho+x_2\|_{H^4})$$
    and $\|\rho+x_2\|_{H^4}\leq C$, and restricting ourselves to times such that $\|W\|_{H^3}\leq 1$ we have
    $$\partial_{t}\|W\|_{H^3}\leq C\|W\|_{H^3}+\frac{C}{N^3}. $$
    This gives us $\|W\|_{H^3}\leq \frac{1}{N^3}\text{e}^{Ct}$, and in particular 
    \begin{equation}\label{CH3W1}
        \partial_{t}\|W\|_{H^3}\leq \frac{C}{N^3}, \quad\|W\|_{H^3}\leq \frac{Ct}{N^3}.
    \end{equation}
    Next, we will compare $\rho_{1}(x,t)=N\rho_{pert,N}(\frac{x}{N})$ and $\rho_{2}(x,t)=N\bar{\rho}_{pert,N}(\frac{x}{N})$, which fulfil the evolution equations
    \begin{align}
    \partial_{t}\rho_{1}(x,t)+(\mathbf{u}[N\rho_{N}(\frac{\cdot}{N},t)](x,t)+\mathbf{u}[\rho_{1}](x,t))\cdot\nabla\rho_{1}(x,t)&=0, \\
    \partial_{t}\rho_{2}(x,t)+(\partial_{x_{1}}u_{1}[\rho_{N}](x=0,t))(x_{1},-x_{2})\cdot\nabla\rho_{2}(x,t)&=0, \\\rho_{1}(x,t=0)=\rho_{2}(x,t=0)&=f(x)\frac{\sin(N^{\frac{1}{10}}x_{1})}{N^{\frac{1}{5}}}.
    \end{align}

    Note that, by the boundedness assumption for $\|\mathbf{u}[\rho_{N}+x_2]\|_{C^1}+\|\mathbf{u}[\rho_{pert,N}]\|_{C^1}$, we know that
    $\text{supp}(\rho_{1}),\text{supp}(\rho_{2})\subset B_{C}(0)$ for some constant $C$ independent of $N$, as long as $t\in[0,\text{min}(t_{N},T)]$. Then, similarly as before, defining $\rho_{1}=W_{pert}+\rho_{2}$, we have that
    \begin{align*}
        \partial_{t}W_{pert}&+\left[\mathbf{u}[N\rho_{N}(\frac{\cdot}{N},t)](x,t)-\partial_{x_{1}}u_{1}[\rho_{N}(\cdot,t)](x=0)(x_{1},-x_{2})\right]\cdot\nabla\rho_{2}\\
        &+\left[\mathbf{u}[N\rho_{N}(\frac{\cdot}{N},t)](x,t)\right]\cdot\nabla W_{pert}+\left[\mathbf{u}[\rho_{2}+W_{pert}]\right]\cdot\nabla (W_{pert}+\rho_{2})=0.
    \end{align*}
    By taking derivatives in the evolution equation we can obtain
    \begin{align}\label{WpertH3}
    \partial_{t}\|W_{pert}\|^2_{H^3}\leq& C\|W_{pert}\|_{H^3}[\|W_{pert}\|^2_{H^3}+\sum_{j=0}^{3}(\|W_{pert}\|_{H^{3-j}}\|\rho_{2}\|_{C^{1+j}})\nonumber\\
    & \quad +\sum_{j=0}^{2}\|W_{pert}\|_{H^{3-j}}\|\mathbf{u}[\rho_{2}]\|_{C^{1+j}}+\|F\|_{H^{3}}\nonumber \\
      & \quad +\sum_{i=1,2}\|W_{pert}\|_{H^{3}} \|(\partial_{x_{i}}u_{i}[N\rho_{N}(\frac{\cdot}{N},t)](x,t) \|_{C^3(\text{supp}(W_{pert}))}],
    \end{align}
    where
    $$F:=\mathbf{u}[N\rho_{N}(\frac{\cdot}{N},t)](x,t)-\partial_{x_{1}}u_{1}[\rho_{N}(\cdot,t)](x=0)(x_{1},-x_{2})]\cdot\nabla\rho_{2}+\mathbf{u}[\rho_{2}]\cdot\nabla \rho_{2}.$$
    Similarly, in $L^2$ we have
    \begin{align}\label{L2W}\partial_{t}\|W_{pert}\|^2_{L^2}\leq& C\|W_{pert}\|_{L^2}[\|W_{pert}\|_{L^2}\|\rho_{2}\|_{C^1}+\|F\|_{L^2}].
    \end{align}
    To bound $F$, we first note that, if we define 
    \begin{align}\label{eq:D}
    D(t)=\text{e}^{-\int_{0}^{t}\partial_{x_{1}}u_{1}[\rho_{N}(\cdot,t)](x=0)},
    \end{align}
    then
    $$\rho_{2}(x,t)=f(D(t)x_{1},D(t)^{-1}x_{2})\frac{\sin(D(t)N^{\frac{1}{10}}x_{1})}{N^{\frac{1}{5}}}$$
    and since $D(t)\leq \text{e}^{M}$, we can apply Lemma \ref{boundvelocity}, the bounds for the support of $\rho_{2}$ and direct computations to get, for $s \in [0,3]$,
    \begin{align*}
    \|u_{1}[\rho_{2}]\partial_{x_{1}}\rho_{2}\|_{H^s}&\leq \sum_{i=0}^{s}\|u_{1}[\rho_{2}]\|_{C^{i}}\|\rho_{2}\|_{H^{s+1-i}}\leq \frac{C\ln(N)}{N^{\frac{4-s}{10}}}\\
    \|u_{2}[\rho_{2}]\partial_{x_{2}}\rho_{2}\|_{H^s}&\leq \sum_{i=0}^{s}\|u_{1}[\rho_{2}]\|_{C^{i}}\|\rho_{2}\|_{H^{s+1-i}}\leq \frac{C\ln(N)}{N^{\frac{4-s}{10}}}\\
    \|\mathbf{u}[\rho_{2}]\cdot\nabla\rho_{2}\|_{H^s}&\leq \frac{C_{s}\ln(N)}{N^{\frac{4-s}{10}}}.
    \end{align*}

    Furthermore, using the bounds for the support of $\rho_{N}(x,t)$, for $j=0,1,...$
    \begin{align*}
        \left\|\frac{\partial^{j} }{\partial x_{1}^{j-l}\partial x_{2}^{l}}\mathbf{u}[N\rho_{N}(\frac{\cdot}{N},t)+x_2](x,t)\right\|_{L^{\infty}(B_{R}(0))}&\leq \left\|N^{1-j}\frac{\partial^{j} }{\partial x_{1}^{j-l}\partial x_{2}^{l}}\mathbf{u}[\rho_{N}(\cdot,t)+x_2](x,t)\right\|_{L^{\infty}(B_{\frac{R}{N}}(0))}\\
        &\leq C_{R,j}N^{1-j}\|\rho_{N}(\cdot,t)+x_2\|_{L^2}\leq \frac{C_{R,j}}{N^{j-1}},
    \end{align*}

    and since $\rho_{N}$ is symmetric, $\mathbf{u}[\rho_{N}](x=0)=0$,
    $$\|[\mathbf{u}[N\rho_{N}(\frac{\cdot}{N},t)+x_2)](x,t)-\partial_{x_{1}}u_{1}[\rho_{N}(\cdot,t)+x_2](x=0)(x_{1},-x_{2})]\|_{C^{3}(B_{R}(0))}\leq \frac{C}{N}.$$
    Since we also have that
    $$\|\rho_{2}\|_{C^j}\leq \frac{C_{j}}{N^{\frac{2-j}{10}}}$$
    using the bounds for the support of $\rho_{2}$ we get, for $s\in[0,3]$,
    $$\left\|[\mathbf{u}[N\rho_{N}(\frac{\cdot}{N},t)](x,t)-\partial_{x_{1}}u_{1}[\rho_{N}(\cdot,t)](x=0)(x_{1},-x_{2})]\cdot\nabla\rho_{2}\right\|_{H^s}\leq C N^{-\frac{1-s}{10}-1}.$$
    Combining these bounds together, we have, for $s\in[0,3]$,
    $$\|F\|_{H^s}\leq C\ln(N) N^{-\frac{4-s}{10}}.$$
    Then, applying \eqref{L2W} together with the bounds for $\rho_{2}$, we have that, for $t\in [0,t_{N}]$
    $$\|W_{pert}\|_{L^2}\leq \frac{Ct\ln(N)}{N^{\frac{4}{10}}}, \quad \partial_{t}\|W_{pert}\|_{L^2}\leq \frac{C\ln(N)}{N^{\frac{4}{10}}}.$$
    Now, since for $t\in[0,t_{N}]$ we have $\|W_{pert}\|_{H^3}\leq 1$, we can apply the bound for $W_{pert}$ in $L^2$, the interpolation inequality and the bounds for $\rho_{2}$ to  the evolution equation for $\|W_{pert}\|_{H^3}$ \eqref{WpertH3} to get
    \begin{align*}
    \partial_{t}\|W_{pert}\|_{H^3}\leq& C\|W_{pert}\|_{H^3}+\frac{C\ln(N)}{N^{\frac{1}{10}}}
    \end{align*}
    so that 
    \begin{equation}\label{CH3W2}
        \|W_{pert}\|_{H^3}\leq \frac{Ct\ln(N)}{N^{\frac{1}{10}}},\quad \partial_{t}\|W_{pert}\|_{H^3}\leq \frac{C\ln(N)}{N^{\frac{1}{10}}}.
    \end{equation}
    Combining \eqref{CH3W1} and \eqref{CH3W2} and taking $N$ big we have that $t_{N}\geq T$, and thus the bounds \eqref{CH3W1} and \eqref{CH3W2} are true for the times we consider.
    
    We now note that, since $\rho_{new,N}$ is a solution to \eqref{eq:ipm} with initial data of the form \eqref{eq:data-stable}, we have that
    \begin{align*}
        \de_{t}(\rho_{new,N}+x_2) + \mathbf{u}[ \rho_{new,N}+x_2]\cdot \nabla (\rho_{new,N}+x_2) - u_2[\rho_{new,N}+x_2]=0,
    \end{align*}
    and so
    $$\de_{t}\|\rho_{new,N}+x_2\|_{H^4}\leq C (\|\rho_{new,N}+x_2\|_{H^4}\|\rho_{new,N}+x_2\|_{H^3}+\|\rho_{new,N}+x_2\|_{H^4})$$
    so in particular the bounds we have obtained for the $H^3$ norm of $\rho_{new,N}$ imply that it also belongs to $H^4$ for $t\in[0,T]$. 
    We are now ready to start showing the properties of $\rho_{new,N}$. First, we note that, if we choose $f(x)$ such that $\|f(x)\|_{L^2}=\frac{\epsilon_{2}}{2}$, then we have
    $$\|\rho_{pert,N}(x,0)\|_{H^2}=\left\|f(Nx)\frac{\sin(N^{1+\frac{1}{10}}x_1)}{N^{1+\frac{1}{5}}}\right\|_{H^2}\leq \frac{\epsilon_{2}}{2}+\frac{C}{N^{\frac{1}{10}}},$$
    so that taking $N$ big enough gives us $\|\rho_{pert,N}(x,0)\|_{H^2}\leq \epsilon_{2}$ as desired. Furthermore, we have
    \begin{align*}
        \|\rho_{new,N}(x,t)-\rho(x,t)\|_{C^1}&\leq \|\rho_{N}(x,t)-\rho(x,t)\|_{C^1}+\|\rho_{pert,N}(x,t)-\bar{\rho}_{pert,N}(x,t)\|_{C^1}+\|\bar{\rho}_{pert,N}(x,t)\|_{C^1}\\
        & \leq \|\rho_{N}(x,t)-\rho(x,t)\|_{H^3}+\left\|N\rho_{pert,N}(\frac{x}{N},t)-\bar{\rho}_{pert,N}(\frac{x}{N},t)\right\|_{H^3}+\frac{C}{N^{\frac{1}{10}}}\\
        &\leq \frac{C\ln(N)}{N^{\frac{1}{10}}},
    \end{align*}
    so taking $N$ big we get $\|\rho_{new,N}-\rho\|_{C^1}\leq \epsilon_{3}$, $\|\rho_{new,N}+x_2\|_{C^1}<K$.
    Next, would like to show that $\de_{t}\de_{x_{1}}u_1[\rho_{new,N}(x,t)](x=0)>-d$.
    Since
    $$|\partial_t \partial_{x_{1}} u_{1}[\rho_{new,N}-\rho]|= |\partial_t\partial_{x_{1}} [u_{1}[\rho_{N}-\rho]+u_{1}[\rho_{pert,N}-\bar{\rho}_{pert,N}]+u_{1}[\bar{\rho}_{pert,N}]]|$$
    \begin{align}
    |\partial_t\partial_{x_{1}} u_{1}[\rho_{N}-\rho]|&\leq \partial_{t}\|\rho_{N}-\rho\|_{H^3}\leq \frac{C}{N^3},\\
    |\partial_t\partial_{x_{1}} u_{1}[\rho_{pert,N}-\bar{\rho}_{pert,N}]|&\leq \partial_{t}\|\rho_{pert,N}-\bar{\rho}_{pert,N}\|_{H^3}\leq \frac{C\ln(N)}{N^{\frac{1}{10}}},
    \end{align}
    which are as small as we want, we only need to worry about $\partial_{t}\partial_{x_{1}} u_{1}[\bar{\rho}_{pert,N}(\cdot,t)]$.
    But, using that
    \begin{align}
    \partial_{t}\partial_{x_{1}} u_{1}[\bar{\rho}_{pert,N}(\cdot,t)](x=0)& =\partial_{x_{1}} u_{1}[\partial_{t}\bar{\rho}_{pert,N}(\cdot,t)](x=0),
    \end{align}
    it is enough to obtain bounds for $\|u_{1}[\de_{t}  \bar{\rho}_{pert,N}(\cdot,t)]\|_{C^1}$.
    But, since 
    $$\de_{t} \bar{\rho}_{pert,N}(\cdot,t)=-\left[k(t)(x_{1},-x_{2})\cdot\nabla\left(f(ND(t)x_{1},ND(t)^{-1}x_{2})\frac{\sin(D(t)N^{1+\frac{1}{10}}x_{1})}{N^{1+\frac{1}{5}}}\right)\right]$$
    and using
    \begin{align*}
        \|u_{1}(\de_{t} \bar{\rho}_{pert,N}(\cdot,t))\|_{C^1}&\leq \left\|Nu_{1}(\de_{t} \bar{\rho}_{pert,N}(\cdot,t))\left(\frac{x}{N},t\right)\right\|_{C^1}\\
        &=\left\|u_{1}\left(\left[k(t)(x_{1},-x_{2})\cdot\nabla\left(f(D(t)x_{1},D(t)^{-1}x_{2})\frac{\sin(D(t)N^{\frac{1}{10}}x_{1})}{N^{\frac{1}{5}}}\right)\right]\right)\right\|_{C^1}
    \end{align*}
    which can be written as the sum of functions of the form $h(x)\sin(D(t)N^{\frac{1}{10}}x_{1}+\theta_{0})$, $\|h(x)\|_{C^1}\leq C$, we can apply Lemma \ref{boundvelocity} to get
    $$\|u_{1}(\de_{t} \bar{\rho}_{pert,N}(\cdot,t))\|_{C^1}\leq \frac{C\ln(N)}{N^{\frac{1}{10}}}.$$
    With all these properties we get
    $$\de_{t}\de_{x_{1}}u_{1}[\rho_{new,N}](x=0,t)> \partial_t \partial_{x_{1}} u_{1}[\rho](x=0,t)-|\partial_t \partial_{x_{1}} u_{1}[\rho_{new,N}-\rho](x=0,t)|>\partial_t \partial_{x_{1}} u_{1}[\rho](x=0,t)-\frac{C\ln(N)}{N^{\frac{1}{10}}},$$
    and using that $\partial_t \partial_{ x_{1}} u_{1}[\rho](x=0,t)>-d$ and the continuity in time of $\partial_t \partial_{x_{1}} u_{1}[\rho](x=0,t)$ for $t\in[0,T]$ we get, for big $N$,
    $$\de_{t}\de_{x_{1}}u_{1}[\rho_{new,N}](x=0,t)> \partial_t \partial_{x_{1}} u_{1}[\rho](x=0,t)-|\partial_t \partial_{x_{1}} u_{1}[\rho_{new,N}-\rho](x=0,t)|>-d.$$
    With this, we have all the conditions in \eqref{growth2}.
    Furthermore, since 
    $$|\de_{t}\de_{x_{1}}u_{1}[\rho_{new,N}](x=0,t)- \partial_t \partial_{x_{1}} u_{1}[\rho](x=0,t)|\leq \frac{C\ln(N)}{N^{\frac{1}{10}}},$$
    by taking $N$ big we obtain \eqref{growth1}.
    To finish, we note that, for $N$ big,
    \begin{align*}
        \|\rho_{new,N}+x_2\|_{H^2}&\geq \|\rho_{pert,N}\|_{H^2}\geq \|\bar{\rho}_{pert,N}\|_{H^2}-\|\bar{\rho}_{pert,N}-\rho_{pert,N}\|_{H^2}\\
        &=\|\bar{\rho}_{pert,N}\|_{H^2}-\left\|N\bar{\rho}_{pert,N}(\frac{x}{N},t)-N\rho_{pert,N}(\frac{x}{N},t)\right\|_{H^2}\\
        &\geq \|\bar{\rho}_{pert,N}\|_{H^2}-\frac{C}{N^{\frac{1}{10}}}\\
        &\geq \frac{3}{4}\|f(D(t)x_{1},D(t)^{-1}x_{2})\de^2_{x_{1}}(\sin(ND(t) x_{1}))\|_{L^2}-\frac{C}{N^{\frac{1}{10}}}\\
        & \geq
        \frac{3\epsilon_{2}}{8}D(t)^2-\frac{C}{N^{\frac{1}{10}}}\geq \frac{\epsilon_{2}}{4}D(t)^2,
    \end{align*}
where in the second-to-last line we used scaling invariance,
so that the only remaining thing is to obtain lower bounds for $$D(t)=\text{e}^{-\int_{0}^{t}k(s)\, ds},$$ with $k(t)=(\de_{x_{1}}u_{1}[\rho(\cdot,t))](x=0)$. Using the bound for the derivative of $k(t)$, we have that $k(t)+td$ is monotone increasing, and therefore
$$-\int_{0}^{t}(k(s)+sd)\, ds\geq -\frac{t}{T}\int_{0}^{T}(k(s)+sd)\, ds=\frac{t}{T}(M-\frac{T^2d}{2})$$
where we used that, if $g(t)$ is monotone decreasing, then, for $t_{2}\geq t_{1}$
$$\frac{1}{t_{1}}\int_{0}^{t_{1}}g(s)\, ds\geq \frac{1}{t_{2}}\int_{0}^{t_{2}}g(s)\, ds,$$
and thus
$$D(t)\geq \text{e}^{\frac{tM}{T}-\frac{tTd}{2}}$$
as we wanted to prove. Finally, to prove \eqref{convh2K}, we note that
$$\|\rho(x,t)-\rho_{new}\|_{C^{2.5}(\R^2\setminus B_{\delta})}=\|\rho(x,t)-\rho_{N}(x,t)+\rho_{pert,N}(x,t)\|_{C^{2.5}(\R^2\setminus B_{\delta})}=\|\rho(x,t)-\rho_{N}(x,t)\|_{C^{2.5}(\R^2\setminus B_{\delta})}$$
where we used that $\rho_{pert,N}$ is supported in $B_{\delta}$ for $N$ big (using \eqref{supportpertN}). But we already showed in \eqref{CH3W2} that
\begin{align*}
    \|\rho(x,t)-\rho_{N}(x,t)\|_{H^3(\R^2)}&\leq\frac{Ct\ln(N)}{N^{\frac{1}{10}}};\\
    \|\rho_{new}(x,t)-\rho(x,t)\|_{H^{4}}&=\|(\rho_{new}(x,t)+x_{2})-(\rho(x,t)+x_{2})\|_{H^{4}}\leq C,
\end{align*}
and thus taking $N$ big and using interpolation and Sobolev embedding gives us the desired bound in $C^{2.5}$. 
As for the $L^2$ norm, we just use
$$\|\rho_{new}(x,t)-\rho(x,t)\|_{L^2}\leq \|\rho(x,t)-\rho_{N}(x,t)\|_{H^3(\R^2)}+\|\rho_{pert,N}\|_{L^2}\leq \frac{Ct\ln(N)}{N^{\frac{1}{10}}}$$
which finishes the proof.

\end{proof}
\subsection{Proof of Theorem \ref{thm:main}}
    We will construct our solution using an iterative procedure 
    $$\rho(x,t)=\text{lim}_{n\rightarrow\infty}\rho_{n}.$$
    We start by choosing $\rho_{0}(x,t)$. For this, given our value of $\epsilon$, we apply Lemma  \ref{lem:stable-exact-sol} with $\epsilon_{0}=\frac{\epsilon}{4}$.
    
    Note that $\rho_{0}(x,t)+x_{2}\in H^4$ for $t\in[0,T_{\epsilon}]$, $\de_{t} \de_{x_1}u_{1}[\rho_{0}](x,t)>-K_{\epsilon}$, $\|\rho_{0}(x,t)+x_{2}\|_{C^1}<K_{\epsilon}$ for some $K_{\epsilon}>2$, $\de_{x_1}u_{1}[\rho_{0}](x,t)<0$ and $\rho_0(x,t)$ is supported far from the origin and is symmetric.
    We define
    $$M_{0}:=-\int_{0}^{T}\de_{x_1}u_{1}[\rho_{0}](x,t)\, dt$$
    and
    $$\delta_{0}:=\text{sup}\{\delta:\supp(\rho_{0}(x,t))\cap B_{\delta}(0)=\emptyset\}.$$

    We will now proceed to construct our solution inductively. We start by noting that, through the steps of the construction, both $\rho_{n}$ and $\tilde\rho_{n}$ (both to be defined later) are symmetric 
    and that their properties will hold for $t\in[0,T_{\epsilon}]$, so we will omit mentioning this every time.
    Now, given $\rho_{n}$, with
    $$M_{n}:=-\int_{0}^{T_{\epsilon}}\de_{x_1}u_{1}[\rho_{n}](x,t)\, dt,$$
    $$\delta_{n}:=\text{sup}\{\delta:\supp(\rho_{n}(x,t))\cap B_{\delta}(0)=\emptyset\},$$
    we first construct $\tilde{\rho}_{n}$ the following way:
     we first apply Lemma \ref{inductiondeform} with $\epsilon_{1}=\|\rho_{n}\|_{H^2}$, $\epsilon_{2}=\frac{\nu}{n+1}$ ($\nu$ small to be fixed later), $\epsilon_{3}=4^{-n}$, $T=T_{\epsilon}$, $K=d=K_{\epsilon}$, $M=M_{n}$ to obtain $\tilde{\rho}_{n}$ fulfilling
     \begin{align*}
        \tilde{k}_{n}(t):&=\de_{x_{1}}u_1[\tilde{\rho}_{n}](x=0,t)<0,\; 
        \de_{t}\tilde{k}_{n}(t)>-K_{\epsilon}, \\
        \tilde{M}_{n}:&=-\int_{0}^{T_{\epsilon}}\de_{x_1}u_{1}[\tilde{\rho}_{n}](x,t)\geq M_{n}+\frac{\nu}{n+1}c_{M_{n}},\\
        \|\tilde{\rho}_{n}(x,t)+x_2\|_{C^1}&<K_{\epsilon},\; \|\tilde{\rho}_{n}(x,t)-\rho_{n}(x,t)\|_{C^1}\leq \frac{\nu}{n+1},\\
        \|\tilde{\rho}_{n}(x,t)-\rho_{n}(x,t)\|_{C^{2.5}(\R^2\setminus B_{\delta_{n}})}&\leq \frac{1}{4^n},\; \|\tilde{\rho}_{n}(x,t)-\rho_{n}(x,t)\|_{L^{2}(\R^2)}\leq \frac{1}{4^n},   
    \end{align*}
    $$\tilde{\delta}_{n}:=\text{sup}\{\delta:\text{supp}(\tilde{\rho}_{n}(x,t))\cap B_{\delta}(0)=\emptyset\},\quad \text{supp}(\tilde{\rho}_{n}-\rho_{n})\cap \text{supp}(\rho_{n})=\emptyset$$
    \begin{align}\label{eq:initial-data-nu}
    \|\tilde{\rho}_{n}(x,0)-\rho_{n}(x,0)\|_{H^2}\leq \frac{\nu}{n+1},\quad \|\tilde{\rho}_{n}\|_{C^{1}}<K_{\epsilon},
    \end{align}
    with $c_{M}$ the constant in Lemma \ref{inductiondeform}. 
    
    If we just used $\rho_{n+1}=\tilde{\rho}_{n}$, this would allow us to construct a solution generating a velocity producing a very strong deformation at the origin, and an infinite deformation if we take the limit when $n$ tends to $\infty$. However, this is not enough to ensure growth of the $H^2$ norm. To make sure that the solution will grow appropriately, we add an extra perturbation $\rho_{n+1}=\tilde{\rho}_n+\rho_{pert}$ using Lemma \ref{inductiongrowth}, with this perturbation growing rapidly in $H^2$, which will ensure the desired growth.

    However, to ensure that the added perturbation becomes large in \(H^2\), the deformation generated by \(\tilde{\rho}_{n}\) must be sufficiently large.
     In order to do this, we will only add this perturbation when the deformation generated by $\tilde{\rho}_{n}$ has grown enough, in particular we will add this perturbation every time $\text{e}^{M_{n}}$ becomes bigger than $4^{k^2}$ for some $k$.
    
    With this in mind, if there is no $k\in \N$ such that $\text{e}^{M_{n}}\geq 4^{k^2}$, $\text{e}^{M_{i}}<4^{k^2}$ for $i=0,1,..i,n-1$, we just define $\rho_{n+1}(x,t)=\tilde{\rho}_{n}(x,t)$. 
    
    If there is $k\in \N$ such that $\text{e}^{M_{n}}\geq 4^{k^2}$, $\text{e}^{M_{i}}<4^{k^2}$ for $i=0,1,..i,n-1$ (we consider the biggest value if there is more than one $k$ with these properties), then we can apply Lemma \ref{inductiongrowth} with $\epsilon_{1}=\|\tilde{\rho}_{n}\|_{H^2}$, $\epsilon_{2}=\frac{\epsilon}{8}2^{-k}$, $\epsilon_{3}=4^{-n}$, $T=T_{\epsilon}$, $K=d=K_{\epsilon}$, $M=\tilde{M}_{n}$ to obtain $\rho_{n+1}$ fulfilling

  \begin{align*} 
    k_{n+1}(t):&=\de_{x_{1}}u_1[\rho_{n+1}](x=0,t)<0, \; \de_{t}k_{n+1}(t)>-K_{\epsilon},\\
    M_{n+1}:&=-\int_{0}^{T}\de_{1}u_{1}(\rho_{n+1}(x,t))\geq M_{n}-4^{-n}+\frac{\nu}{n+1}c_{M_{n}},\\
    \|\rho_{n+1}(x,t)+x_{2}\|_{C^1}&<K_{\epsilon}, \|\rho_{n+1}(x,t)-\tilde{\rho}_{n}(x,t)\|_{C^1}\leq 4^{-n},\\
    \|\rho_{n+1}(x,t)-\tilde{\rho}_{n}(x,t)\|_{C^{2.5}(\R^2\setminus B_{\tilde{\delta}_{n}})}&\leq 4^{-n},\|\rho_{n+1}(x,t)-\tilde{\rho}_{n}(x,t)\|_{L^{2}(\R^2)}\leq 4^{-n},
    \end{align*}
    $$\delta_{n+1}:=\text{sup}\{\delta:\text{supp}(\rho_{n+1}(x,t))\cap B_{\delta}(0)=\emptyset\}, \text{supp}(\rho_{n+1}-\tilde{\rho}_{n})\cap \text{supp}\tilde{\rho}_{n}=\emptyset$$
    \begin{align}\label{eq:initial-data-eps}
    \|\tilde{\rho}_{n}(x,0)-\rho_{n+1}(x,0)\|_{H^2}\leq \frac{\epsilon}{8}2^{-k},
    \end{align}
    and using \eqref{growth3} we get
    \begin{equation*}
    \|\rho_{n+1}(x,t)\|_{H^2}\geq \epsilon\frac{2^{-k}}{40}\text{e}^{\frac{2M_{n}t}{T}-TtK_{\epsilon}}.
\end{equation*}
Note that, in particular, we always have that
\begin{align*}
    \|\rho_{n+1}(x,t)-\rho_{n}(x,t)\|_{C^{2.5}(\R^2\setminus B_{\delta_{n}})}&\leq \frac{2}{4^{n}},\\
    \|\rho_{n+1}(x,t)-\rho_{n}(x,t)\|_{L^{2}(\R^2)}&\leq \frac{2}{4^{n}},
\end{align*}
and, denoting by $D^{1}$ a generic derivative
$$|D^{1}\rho_{n+1}(x=x_{0},t)-D^{1}\rho_{n}(x=x_{0},t)|\leq \begin{cases}
\frac{2}{4^{n}},\quad   x_{0}\in \R^2\setminus B_{\delta_{n}}\\
\frac{2}{n+1},\quad x_{0}\in B_{\delta_{n}}\setminus B_{\delta_{n+1}}\\
0,\quad  x_{0}\in B_{\delta_{n+1}}
\end{cases}$$
which implies that, for $n_{2}\geq n_{1}$
$$\|\rho_{n_{2}}(x,t)-\rho_{n_{1}}(x,t)\|_{C^{1}}\leq \frac{2}{n_{1}+1}+\sum_{j=n_{1}}^{\infty}\frac{2}{4^{j}} \; $$
and thus it is a Cauchy sequence in $C^1$.
To finish the proof we now show that the limit function $\rho_{\infty}(x,t)=\text{lim}_{n\rightarrow\infty}\rho_{n}(x,t)$ exists and has all the desired properties.
    
    \textbf{Step 1: Unboundedness of $M_{n}$.}
    
    We start by showing that the sequence given by $M_{n}$ tends to infinity. For this we can argue by contradiction: Assume that
    $$\text{lim}_{n\rightarrow\infty}M_{n}=M_{\infty}<\infty.$$
    Note that since $M_{n}+\sum_{j=1}^{n}4^{-j}$ is a monotone increasing sequence then either $M_{\infty}\in \R$ exists or $M_{n}$ tends to infinity. Moreover, we know from Lemma \ref{inductiondeform} that $c_{M_n}>0$ is a constant that depends only on $M_n$ and decreases as $M_n$ increases. 
    Then, we have that
    $$M_{n}\geq M_{0}+\sum_{j=0}^{n-1}\frac{1}{j+1}c_{M_{j}}-\sum_{i=0}^{\infty}4^{-k}\geq c_{M_{\infty}}\sum_{j=1}^{n}\frac{1}{j}-2,$$
    but since by hypothesis $c_{M_{\infty}}$ is bounded from below and the harmonic series is divergent, then we must have $M_{\infty}=\infty$.

   \textbf{Step 2: Convergence of the sequence and upper bounds}
   
   Let us now consider $\rho_{\infty}$, the limit in $C^1$.
   If we consider $n_{2}\geq n_{1}$ fulfilling $\text{e}^{M_{n_{1}}}\geq 4^{k^2}$, using that, for $n\geq j$ 
   $$\text{supp}(\rho_{n+1}(x,0)-\rho_{n}(x,0))\cap\text{supp}(\rho_{j}(x,0))=\emptyset$$
   which implies
    $$\text{supp}(\rho_{n+1}(x,0)-\rho_{n}(x,0))\cap\text{supp}(\rho_{j}(x,0)-\rho_{j-1}(x,0))=\emptyset$$
   we have, from \eqref{eq:initial-data-nu}-\eqref{eq:initial-data-eps}, that
   $$\|\rho_{n_{1}}(x,0)-\rho_{n_{2}}(x,0)\|^2_{H^2}= \sum_{j=n_{1}}^{n_{2}-1}\|\rho_{j+1}(x,0)-\rho_{j}(x,0)\|^2_{H^2}\leq \sum_{n_{1}}^{n_{2}}\frac{\nu^2}{(j+1)^2}+\sum_{i=k}^{\infty}\frac{\epsilon^2}{64}4^{-i}$$
   which tends to $0$ as $n_{1},k$ tend to infinity, and therefore
   $$\text{lim}_{n\rightarrow\infty}\|\rho_{\infty}(x,0)-\rho_{n}(x,0)\|_{H^2}=0,$$
   so in particular, using again the separation of the supports at initial time
   $$\|\rho_{\infty}(x,0)\|^2_{H^2}= \|\rho_{\text{in}}(x,0)\|^2_{H^2}+\sum_{j=0}^{\infty}\|\rho_{j+1}(x,0)-\rho_{j}(x,0)\|^2_{H^2}\leq \frac{\epsilon^2}{16}+\sum_{j=0}^{\infty}\frac{\nu^2}{(j+1)^2}+\sum_{i=1}\frac{\epsilon^2}{64}4^{-i}$$
   so by taking $\nu$ small enough, we get $\|\rho_{\infty}(x,0)\|_{H^2}\leq \epsilon.$
   Finally, note that, for any $n_{2}\geq n_{1}\geq n_{0}$ we have that
   $$\|\rho_{n_{2}}(x,t)-\rho_{n_{1}}(x,t)\|_{C^{2.5}(\R^2\setminus  B_{\delta_{n_{0}}}(0))}\leq \sum_{i=n_{1}}^{\infty}\frac{\epsilon}{8}4^{-i}$$
   and thus
   $$\text{lim}_{n\rightarrow\infty}\|\rho_{\infty}(x,t)-\rho_{n}(x,t)\|_{C^{2.5}(\R^2\setminus B_{\delta_{n_{0}}}(0))}=0.$$

    \textbf{Step 3: existence of the solution $\rho_{\infty}$.}
    
    Using the convergence of $\rho_{n}+x_{2}$ in $C^1$ and $L^2$ we get that
    $$\text{lim}_{n\rightarrow\infty}\|\mathbf{u}[\rho_{\infty}]-\mathbf{u}[\rho_{n}]\|_{C^{\alpha}}=\text{lim}_{n\rightarrow\infty}\|\mathbf{u}[\rho_{\infty}]-\mathbf{u}[\rho_{n}]\|_{C^{\alpha}}=0$$
    for $\alpha<1$.
    In particular, integrating in time and passing to the limit we get
    $$\rho_{\infty}(t_{2})-\rho_{\infty}(t_{1})=-\int_{t_{1}}^{t_{2}}\mathbf{u}[\rho_{\infty}]\cdot\nabla \rho_{\infty}dt.$$
    Furthermore, since we actually have 
    $$\text{lim}_{n\rightarrow\infty}\|\rho_{\infty}(x,t)-\rho_{n}(x,t)\|_{C^{2.5}(\R^2\setminus B_{\delta})}$$
    for any $\delta>0$, this combined with the $L^2$ convergence of $\rho_{n}$ gives us
    $$\text{lim}_{n\rightarrow\infty}\|\mathbf{u}[\rho_{\infty}](x,t)-\mathbf{u}[\rho_{n}](x,t)\|_{C^{2}(\R^2\setminus B_{\delta})}=0.$$
    Note that we actually get better convergence, but it is enough for this proof to obtain it in $C^2$.
    With this, after taking a spatial derivative in the evolution equation, integrating in time and passing to the limit
    $$\partial_{x_{i}}\rho_{\infty}(t_{2})-\partial_{x_{i}}\rho_{\infty}(t_{1})=-\int_{t_{1}}^{t_{2}}\partial_{x_{i}}(\mathbf{u}[\rho_{\infty}]\cdot\nabla \rho_{\infty})\, dt,$$
    for any $x\neq 0$.
    With this we get that, for $x\neq 0$, both $\rho_{\infty}$ and $\partial_{x_{i}}\rho_{\infty}$ are Lipschitz in time, which in particular also gives us continuity in time for $\mathbf{u}[\rho_{\infty}]$, and thus we can pass to the limit $t_{2}\rightarrow t_{1}$ and obtain, for $x\neq 0$

    $$\partial_{t}\rho_{\infty}(x,t)=-\mathbf{u}[\rho_{\infty}](x,t)\cdot\nabla \rho_{\infty}(x,t).$$

    For $x=0$, we just use that $\rho_{\infty}(x=0,t)=\mathbf{u}[\rho_{\infty}](x=0,t)=0$, $\rho_\infty \in C^1$, to get that
    $$\partial_{t}\rho_{\infty}(x=0,t)=0=-\mathbf{u}[\rho_{\infty}](x=0,t)\cdot\nabla \rho_{\infty}(x=0,t),$$
    and thus $\rho_{\infty}$ is a classical solution to IPM.
    
    \textbf{Step 3: Properties of the limit solution.}
   Note that, since we have showed the initial smallness in $H^2$ of $\rho_{\infty}$ plus the fact that it is a solution, it is enough to now show that, for any $t_{0}\in(0,T_{\epsilon})$
   $$\|\rho_{\infty}(x,t_{0})+x_{2}\|_{H^2}=\infty.$$
Given some $k\geq 1$, if we consider $n$  such that $\text{e}^{M_{n}}\geq 4^{k^2}$, $\text{e}^{M_{i}}<4^{k^2}$ for $i=0,1,..i,n-1$, then
    \begin{align*}
        \|\rho_{\infty}+x_{2}\|_{H^2(\R^2\setminus B_{\delta_{n+1}})}&\geq \|\rho_{n+1}+x_{2}\|_{H^2(\R^2\setminus B_{\delta_{n+1}})}-\sum_{j=n+1}^{\infty}\|\rho_{j+1}-\rho_{j}\|_{H^2(\R^2\setminus B_{\delta_{n+1}})}\\
        &\geq \epsilon\frac{2^{-k}}{40}\text{e}^{\frac{2M_{n}t}{T}-Tt}-\sum_{j=n}\frac{2}{4^{n+1}}\geq  \epsilon\frac{4^{\frac{2kt}{T}-\frac12}}{40}\text{e}^{-T^2}-1
    \end{align*}
    so in particular, for any $t_{0}\in (0,T)$, by taking $k$ to infinity we get
    \begin{align*}
        &\|\rho_{\infty}(x,t_{0})+x_{2}\|_{H^2}\geq \text{lim}_{k\rightarrow\infty}c\epsilon4^{\frac{2kt}{T}-\frac12}-1=\infty.
    \end{align*}
\textbf{Step 4: Uniqueness.}
To show uniqueness, we just assume that another solution $\tilde{\rho}_{\infty}$ exists, and defining $W(x,t):=-\rho_{\infty}(x,t)+\tilde{\rho}_{\infty}(x,t)$, subtracting the evolution equations for $\rho_{\infty}$ and $\tilde{\rho}_{\infty}$ we get
$$\partial_{t}W=-\mathbf{u}[W+\rho_{\infty}]\cdot\nabla W-\mathbf{u}[W]\cdot\nabla \rho_{\infty}-u_{2}[W]$$
and thus
$$\partial_{t}\|W\|_{L^2}\leq C(1+\|\rho_{\infty}\|_{C^1})\|W\|_{L^2}$$
and a Grönwall inequality gives us $\|W\|_{L^2}=0$, which finishes the proof.


\section*{Acknowledgment}
RB is partially supported by the Italian Ministry of University and Research, PRIN 2020 entitled ``PDEs, fluid dynamics and transport equation'' and PRIN 2022HSSYPN ``Turbulent Effects vs Stability in Equations from Oceanography'', PNRR Italia Domani, funded
by the European Union under NextGenerationEU, CUP B53D23009300001.
RB warmly thanks Instituto de Ciencias Matem\'aticas, where part of this work was elaborated. This work is supported in part by the Spanish Ministry of Science
and Innovation, through the “Severo Ochoa Programme for Centres of Excellence in R$\&$D (CEX2019-000904-S \& CEX2023-001347-S)” and 152878NB-I00. We were also partially supported by the ERC Advanced Grant 788250, and by the SNF grant FLUTURA: Fluids, Turbulence, Advection No. 212573.
\bibliographystyle{siam}

\bibliography{biblio}

\end{document}